\documentclass{amsart}

\usepackage{amsmath,amsthm,amssymb,fancyhdr,graphicx,bbm,cancel,color,mathrsfs,todonotes,hyperref,xypic, pinlabel}
\usepackage[all]{xy}

\newtheorem{thm}{Theorem}[section]
\newtheorem{lem}[thm]{Lemma}
\newtheorem{prop}[thm]{Proposition}
\newtheorem{cor}[thm]{Corollary}
\theoremstyle{definition} 
\newtheorem{deftn}[thm]{Definition}

\newtheorem{prob}[thm]{Problem} 
\theoremstyle{remark} 
\newtheorem{rmk}[thm]{Remark}
\numberwithin{equation}{section} 

\newcommand{\re}{\mathbb{R}}\newcommand{\q}{\mathbb{Q}}
\newcommand{\co}{\mathbb{C}}\newcommand{\z}{\mathbb{Z}}

\newcommand{\al}{\alpha}\newcommand{\be}{\beta}\newcommand{\ga}{\gamma}\newcommand{\de}{\delta}\newcommand{\ep}{\epsilon}\newcommand{\si}{\sigma}
\newcommand{\vp}{\varphi}\newcommand{\De}{\Delta}\newcommand{\la}{\lambda}
\newcommand{\Ga}{\Gamma}
\newcommand{\what}{\widehat}\newcommand{\wtil}{\widetilde}\newcommand{\ta}{\theta}\newcommand{\om}{\omega}

\newcommand{\cd}{\cdots}\newcommand{\ld}{\ldots}
\newcommand{\sbs}{\subset}\newcommand{\bs}{\backslash}\newcommand{\pa}{\partial}\newcommand{\bbm}{\mathbbm}

\newcommand{\xra}{\xrightarrow}
\newcommand{\car}{\curvearrowright}
\newcommand{\ra}{\rightarrow}
\newcommand{\hra}{\hookrightarrow}

\newcommand{\bb}[1]{\mathbb{#1}}\newcommand{\ca}[1]{\mathcal{#1}}\newcommand{\un}[1]{\underline{#1}}\newcommand{\ov}[1]{\overline{#1}}\newcommand{\mf}{\mathfrak}

\newcommand{\fr}[2]{\frac{#1}{#2}}

\newcommand{\ot}{\otimes}
\newcommand{\lan}{\langle}\newcommand{\ran}{\rangle}
\newcommand{\op}{\oplus}
\newcommand{\ti}{\times}

\newcommand{\aut}{\text{Aut}}

\newcommand{\rest}[2]{#1\bigr\vert_{#2}}
\newcommand{\ad}{\text{ad}}
\newcommand{\hy}{\bb H}\newcommand{\Sl}{\text{SL}}
\newcommand{\SO}{\text{SO}}\newcommand{\psl}{\text{PSL}}
\newcommand{\isom}{\text{Isom}}

\newcommand{\homeo}{\text{Homeo}}
\newcommand{\diff}{\text{Diff}}\newcommand{\SU}{\text{SU}}
\newcommand{\SP}{\text{SP}}

\newcommand{\Mod}{\text{Mod}}
\newcommand{\eu}{\text{eu}}

\newcommand{\Top}{\text{Top}}

\newcommand{\ts}{\bb T}

\newcommand{\Hom}{\text{Hom}}\newcommand{\End}{\text{End}}

\newcommand{\gl}{\mbox{GL}}

\addtocontents{toc}{\protect\setcounter{tocdepth}{1}}

\begin{document}

\title{Cohomological obstructions to Nielsen realization}

\author{Bena Tshishiku}
\address{Department of Mathematics, University of Chicago, Chicago, IL 60615} \email{tshishikub@math.uchicago.edu}


\date{November 15, 2013}


\begin{abstract}
For a based manifold $(M,*)$, the question of whether the surjective homomorphism $\diff(M,*)\ra\pi_0\diff(M,*)$ admits a section is an example of a Nielsen realization problem. This question is related to a question about flat connections on $M$-bundles and is meaningful for $M$ of any dimension. In dimension 2, Bestvina-Church-Souto \cite{bcs} showed a section does not exist when $M$ is closed and has genus $g\ge2$. Their techniques are cohomological and certain aspects are specific to surfaces. We give new cohomological techniques to generalize their result to many locally symmetric manifolds. The main tools include Chern-Weil theory, Milnor-Wood inequalities, and Margulis superrigidity. 
\end{abstract}

\maketitle


\section{Introduction}\label{sec:intro}

Let $M$ be a manifold with basepoint $*\in M$, and let $\diff(M,*)$ denote the group of $C^1$-diffeomorphisms of $M$ that fix $*$. Denote by $\diff(M,*)\ra\pi_0\diff(M,*)$ the natural surjection that sends a diffeomorphism to its isotopy class. The central object in this paper is the \emph{point-pushing homomorphism}
\[\text{Push}: \pi_1(M,*)\ra\pi_0\diff(M,*),\] 
where Push$([\ga])$ is the isotopy class that ``pushes $*$ along $\ga$" (see Section \ref{sec:pushing}).

\vspace{.1in} 

\begin{prob}\label{prob}
Does the projection $\diff(M,*)\ra\pi_0\diff(M,*)$ admit a \emph{section} over the point-pushing homomorphism? In other words, does there exist a homomorphism $\vp:\pi_1(M,*)\ra\diff(M,*)$ so that the following diagram commutes? 
\begin{equation}\label{diag:q1}
\begin{gathered}
\begin{xy}
(-10,-15)*+{\pi_1(M,*)}="A";
(30,0)*+{\diff(M,*)}="B";
(30,-15)*+{\pi_0\diff(M,*)}="C";
{\ar@{-->}"A";"B"}?*!/_3mm/{\vp};
{\ar "B";"C"}?*!/_3mm/{};
{\ar "A";"C"}?*!/^3mm/{\text{Push}};
\end{xy}
\end{gathered}
\end{equation}
\end{prob}
\noindent If $\vp$ exists we say Push is \emph{realized by diffeomorphisms}. 
\vspace{.1in}

Problem \ref{prob} is an instance of the general \emph{Nielsen realization problem}, which asks if a group homomorphism $\Lambda\ra\pi_0\diff(M)$ admits a section $\Lambda\ra\diff(M)$. For example, let $E\ra B$ be an $M$-bundle with associated monodromy representation 
\[\mu:\pi_1(B)\ra\pi_0\diff(M).\]
If $\mu$ is not realized by diffeomorphisms, then $E\ra B$ does not admit a \emph{flat connection}. When $E\ra B$ has a section, $\mu$ factors through $\pi_0\diff(M,*)$, and if $\mu$ is not realized by diffeomorphisms, then $E\ra B$ does not admit a flat connection for which the section is parallel (see Sections \ref{sec:pushing} and \ref{sec:flat}). 

The Nielsen realization problem has a long history, with both positive and negative results. For example, when $M$ is a surface see Kerckhoff \cite{Kerckhoff}, Morita \cite{morita_nonlifting}, Markovic \cite{markovic}, and Franks-Handel \cite{frankshandel}; for examples in higher dimensions see Block-Weinberger \cite{bw} and Giansiracusa \cite{giansiracusa}. 

This paper was inspired by the paper \cite{bcs} of Bestvina-Church-Souto, which shows that Push is not realized by diffeomorphisms when $M$ is a closed surface of genus $g\ge2$. These surfaces are simple examples of (finite volume) locally symmetric manifolds $M=\Ga\bs G/K$, where $G$ is a semisimple real Lie group of noncompact type, $K\sbs G$ is a maximal compact subgroup, and $\Ga\sbs G$ is a torsion-free lattice. We will study Problem \ref{prob} for locally symmetric manifolds and generalize the result in \cite{bcs} for surfaces to large classes of locally symmetric manifolds. 
We will focus on the case when $\Ga\sbs G$ is irreducible. To address Problem \ref{prob}, we use the following basic dichotomy of locally symmetric manifolds. 

\vspace{.1in}

\noindent{\bf Type 1.} {\it  $M$ has at least one nonzero Pontryagin class.} 

\vspace{.1in}

\noindent{\bf Type 2.} {\it  $M$ has trivial Pontryagin classes.} 

\vspace{.1in}

Being Type (1) or Type (2) depends only on the group $G$ when $M=\Ga\bs G/K$ is compact; this is a consequence of the Proportionality Principle (see \cite{kt_flat}, Section 4.14). Following Borel-Hirzebruch \cite{bh}, the author has determined precisely which $G$ yield Type (1) manifolds (this computation is completed in \cite{tshishiku2}). The following table shows which $M$ are Type (1) or (2) when $G$ is simple. For $G$ a product of simple groups, $M$ is Type (1) if and only if at least one factor of $G$ is Type (1).

\vspace{.05in}
\begin{center}
\begin{tabular}{ll|lll}
Type 1&&&Type 2\\[1mm]
\hline &&\\ [-1.5ex]
$\SU(p,q)$&$p,q\ge1$ and $p+q\ge2$&& $\Sl(n,\re)$ &for $n\ge2$\\[1mm]
$\SP(2n,\re)$& $n\ge2$&&$\SO(n,1)$& for $n\ge2$\\[1mm]
$\SO(p,q)$ & $p,q\ge2$ and $(p,q)\neq (2,2)$ or $(3,3)$&&$\SU^*(2n)$&$n\ge2$\\[1mm]
$\SP(p,q)$ & $p,q\ge1$&&$E_{6(-26)}$\\[1mm]
$\SO^*(2n)$&$n\ge3$&&$\Sl(n,\co)$ &for $n\ge2$\\[1mm]
$G_{2(2)}$&& & $\SO(n,\co)$ &for $n\ge2$\\[1mm]
$F_{4(4)}$&&&$\SP(2n,\co)$ &for $n\ge2$\\[1mm]
$F_{4(-20)}$&&&$G_2(\co)$&\\[1mm]
$E_{6(6)}$&&&$F_4(\co)$\\[1mm]
$E_{6(2)}$&&&$E_6(\co)$&\\[1mm]
$E_{6(-14)}$&&&$E_7(\co)$&\\[1mm]
$E_{7(7)}$&&&$E_8(\co)$&\\[1mm]
$E_{7(-5)}$&\\[1mm]
$E_{7(-25)}$&\\[1mm]
$E_{8(8)}$&\\[1mm]
\end{tabular}
\end{center}

\vspace{.05in}

We further divide Type (2) examples according to the real rank of $G$ and the representation theory of $\Ga$.

\vspace{.1in}

\noindent{\bf Type 2a.} {\it $M=\Ga\bs G/K$ has trivial Pontryagin classes, $\re$-\emph{rank}$(G)\ge2$, and every finite dimensional unitary representation of $\Ga$ is virtually trivial.}

\vspace{.1in}

\noindent{\bf Type 2b.} {\it $M=\Ga\bs G/K$ has trivial Pontryagin classes, and if $\re$-\emph{rank}$(G)\ge2$, then $\Ga$ has a unitary representation $\Ga\ra U(n)$ with infinite image.} 

\vspace{.1in} 

Most Type (2) examples are Type (2a). For example, if $M$ is Type (2) and $\q$-rank$(\Ga)\ge1$, then $M$ is Type (2a); see \cite{witte-morris} Theorem 13.3. There are also many Type (2a) $M$ with $\q$-rank$(\Ga)=0$; see Ch.\ 15 in \cite{witte-morris} and the discussion in Section \ref{sec:res} below. 

\subsection{Results} The following theorems are the main results of this paper. 

\begin{thm}\label{thm:pont}
Let $M$ be a Riemannian manifold with nonpositive curvature. Assume that some Pontryagin class of $M$ is nontrivial. Then \emph{Push} is not realized by diffeomorphisms. In particular, if $M=\Ga\bs G/K$ is an irreducible locally symmetric manifold of Type $(1)$, then \emph{Push} is not realized by diffeomorphisms. 
\end{thm}

\begin{thm}\label{thm:main}
Let $G$ be a semisimple real Lie group with no compact factors. Let $K\sbs G$ be a maximal compact subgroup and let $\Ga\sbs G$ be an irreducible lattice. 
If $M=\Ga\bs G/K$ has Type $(2\emph{a})$ then \emph{Push} is not realized by diffeomorphisms. 
\end{thm}

\begin{thm}
\label{thm:surfprod}
Let $G=\emph{\Sl}_2(\re)\ti\cd\ti\emph{\Sl}_2(\re)$, let $\Ga\sbs G$ be a cocompact lattice (possibly reducible), and let $M=\Ga\bs G/K$. Then \emph{Push} is not realized by diffeomorphisms. 
\end{thm} 

\vspace{.1in} 

\noindent{\bf Methods of Proof.} The proofs of Theorems \ref{thm:pont} and \ref{thm:surfprod} have the same skeleton as the proof of the main theorem in \cite{bcs}. The key point is to show 
\begin{itemize}
\item[($\star$)] If $M$ has nonpositive curvature and Push is realized by diffeomorphisms, then the tangent bundle $TM\ra M$ has the same Euler and Pontryagin classes as a bundle with a flat $\gl_n(\re)$ connection. 
\end{itemize}
This amounts to the following commutative diagram (see Section \ref{sec:genproc}).
\begin{equation}\label{diag:maincoho}
\begin{gathered}
\begin{xy}
(-30,-10)*+{H^*(M)}="A";
(-5,0)*+{H^*(BG^\de)}="B";
(30,0)*+{H^*(BG)}="C";
(60,-10)*+{H^*\big(B\homeo(S^{n-1})\big)}="D";
(-5,-20)*+{H^*(B\gl_{n}\re^\de)}="E";
(30,-20)*+{H^*(B\gl_{n}\re)}="F";
{\ar"B";"A"}?*!/^3mm/{a^*};
{\ar "E";"A"}?*!/_3mm/{b^*};
{\ar "C";"B"}?*!/_3mm/{i^*};
{\ar "D";"C"}?*!/^3mm/{c^*};
{\ar "F";"E"}?*!/^3mm/{j^*};
{\ar "D";"F"}?*!/_3mm/{d^*};
\end{xy}
\end{gathered}
\end{equation}
Classical arguments and computations in algebraic topology (involving Chern-Weil theory and Milnor-Wood inequalities) provide examples of $M$ for which $TM\ra M$ has Euler or Pontryagin classes that differ from those of any flat bundle. By ($\star$) Push is not realized for these $M$.  

The proof of Theorem \ref{thm:main} differs from the above outline in an essential way: classical obstruction theory does not apply to manifolds of Type (2a) because they have vanishing Pontryagin classes and the Milnor-Wood inequalities are ineffective. We prove Theorem \ref{thm:main} using a combination of Margulis superrigidity and representation theory of Lie algebras. The rough idea follows. Suppose for a contradiction that Push is realized by diffeomorphisms, and let $\rho_0:\Ga\ra\gl_n(\re)$ be the action on the tangent space of $*$. 
\begin{enumerate}
\item[(i)] Use Margulis superrigidity to extend $\rho_0$ to $G$. The extension we will be denoted $\rho_0:G\ra\gl_n(\re)$.
\item[(ii)] Show that the restriction $\rho_0\rest{}{K}:K\ra\gl_n(\re)$ has the same characteristic classes as the \emph{isotropy representation} $\iota:K\ra\aut(\mf p)$ (Section \ref{sec:kact}). This implies $\rho_0\rest{}{K}$ and $\iota$ are isomorphic representations by Proposition \ref{prop:conjreps}. In particular, since $\rho_0\rest{}{K}$ is the restriction of a representation of $G$, the same must be true of $\iota$. 
\item[(iii)] Show that $\iota$ is not the restriction of any representation of $G$. 
\end{enumerate} 
Steps (ii) and (iii) combine to give the desired contradiction. 

\vspace{.1in}

\noindent{\it Remark.} There are two main features of this paper that distinguish it from \cite{bcs}. The first, mentioned above, is the use of Margulis superrigidity in the case when $M$ has no characteristic classes and classical algebraic topology arguments fail. The second is a new, more general proof of ($\star$). The proof of ($\star$) in \cite{bcs} relies on the well-known theorem of Earle-Eells \cite{ee} that an $M$-bundle $E\ra B$ is uniquely determined by its monodromy $\pi_1(B)\ra\Mod(M)$ when $\dim M=2$. This is false in general for $\dim M>2$, but we account for this in Section \ref{sec:bordant} by proving a general fact about fiberwise bordant sphere bundles.

\vspace{.1in}

\noindent{\bf Structure of the paper.} In Section \ref{sec:pushing}, we define the point-pushing homomorphism. In Sections \ref{sec:flat} and \ref{sec:bordant} we recall the definitions of Euler and Pontryagin classes of topological sphere bundles, define the notion of fiberwise bordant bundles, and explain why fiberwise bordant sphere bundles have the same Euler and Pontryagin classes. In Section \ref{sec:res} we describe how to determine if a Type (2) manifold has Type (2a) or (2b). In Section \ref{sec:genproc}, we adapt the argument of Bestvina-Church-Souto \cite{bcs} that translates Problem \ref{prob} to a problem about bordant sphere bundles.  In Sections \ref{sec:proofs} and \ref{sec:isotropyextend} we prove Theorems \ref{thm:pont}, \ref{thm:main}, and \ref{thm:surfprod}, and in Section \ref{sec:zimmer} we mention an example related to the Zimmer program.  

\subsection{Acknowledgements} 

The author would like to thank D.\ Calegari for suggesting the approach using superrigidity; the author thanks S.\ Weinberger for explaining the relationship between this question and the Zimmer program; the author thanks A.\ Hatcher for explaining the proof of Proposition \ref{prop:ratsurj}; the author is grateful to O.\ Randal-Williams for giving an alternate proof of Proposition \ref{prop:stabequiv}, which greatly simplified the author's original argument. The author is indebted to his advisor B.\ Farb who has been extremely generous with his time, energy, and advice throughout this project. Thanks to  B.\ Farb, K.\ Mann, and W.\ van Limbeek for extensive comments on drafts of this paper, and thanks to T.\ Church, K.\ Mann, D.\ Studenmund, W.\ van Limbeek, and D.\ Witte Morris for several useful conversations. Finally, the author thanks the referee for carefully reading the paper and offering valuable suggestions. 

\section{The point-pushing homomorphism $\pi_1(M)\ra\pi_0(\diff(M,*))$}\label{sec:pushing}

First we define Push. Then we express it as the monodromy of a fiber bundle and relate Problem \ref{prob} to a question about flat connections. 

\subsection{Point-pushing diffeomorphisms} The material of this subsection is well known; for further details, see \cite{fm} Page 101. We begin with a simple example. Let $M\sbs\re^2$ be an annulus whose core curve $\ga\simeq S^1$ is the unit circle. Let $*=(1,0)$ be the basepoint. It is easy to construct a family $f_t:M\ra M$ of diffeomorphisms that push $*$ around the curve $\ga$ and fix the boundary of the annulus pointwise. The ending diffeomorphism $f_{2\pi}$ is an example of a point-pushing diffeomorphism. See the figure below. 
\begin{figure}[h!]
  \centering
    \includegraphics[scale=.5]{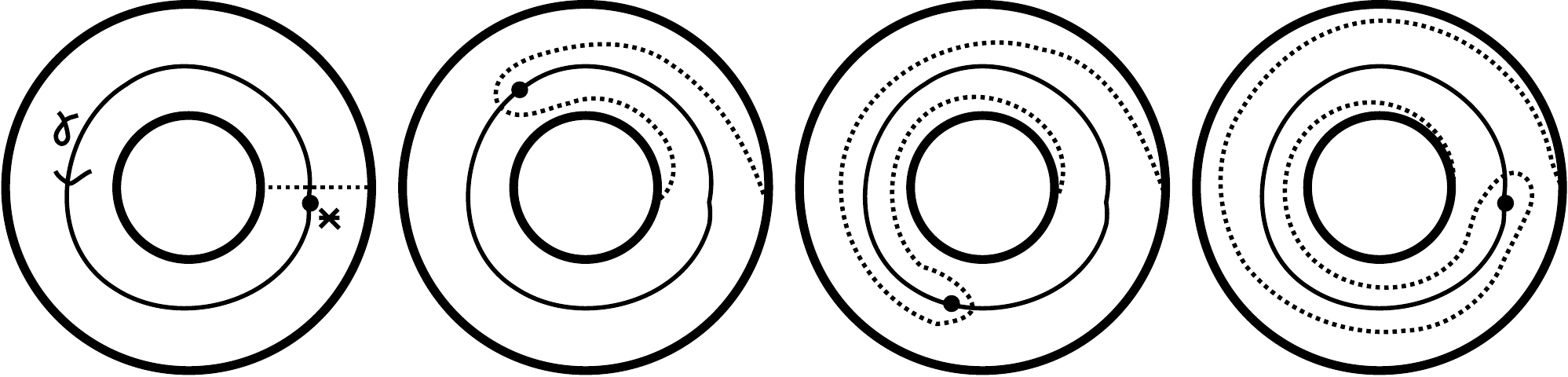}
      \caption{$\ga$ determines a flow under whose time-$2\pi$ map is a point-pushing diffeomorphism.}
\end{figure}

Point-pushing diffeomorphisms can be defined on any based manifold. Let $M$ be a manifold with basepoint $*$, and let $\ga:S^1\ra M$ be a smooth embedding of the unit circle, and assume that $\ga(S^1)$ passes through $*$. The positively oriented unit-speed vector field on $S^1$ defines a vector field on the $\ga(S^1)\sbs M$, and this vector field extends to a vector field on $M$ that is identically zero outside a tubular neighborhood of $\ga(S^1)$. The flow of this vector field is an isotopy of $M$ that moves $*$ around $\ga(S^1)$, and the time-2$\pi$ map is a diffeomorphism $f_\ga:M\ra M$ fixing $*$. The map $f_\ga$ is a \emph{point-pushing diffeomorphism} associated to $\ga$. 

Given $[\ga]\in\pi_1(M)$, there are several choices involved in defining $f_\ga$; however, different choices result in isotopic diffeomorphisms. In particular, if $\ga$ and $\eta$ are loops that are homotopic relative to $*$, then $f_\ga$ and $f_\eta$ are isotopic relative to $*$. In other words, there is a well-defined homomorphism 
\[\begin{array}{lrllll}\text{Push}: &\pi_1(M,*)&\ra&\pi_0\diff(M,*)\\
&[\ga]&\mapsto &[f_\ga]
\end{array}\]
Since $f_\ga$ is defined using a flow, $f_\ga$ is isotopic to the identity; in other words, $[f_\ga]$ is trivial in $\pi_0\diff(M)$. On the other hand, $f_\ga$ is not necessarily isotopic to the identity through diffeomorphisms fixing $*$, and the isotopy class of $f_\ga$ in $\pi_0\diff(M,*)$ can be interesting. In fact, the kernel of the Push map is contained in the center of $\pi_1(M)$. This fact is best viewed from the following perspective. Consider the evaluation map
\[\begin{array}{rcc}\diff(M)&\xra{\eta}&M\\f&\mapsto &f(*).\end{array}\]
The preimage of $*\in M$ is $\diff(M,*)\sbs\diff(M)$. It is known that $\eta$ is a bundle map (see the proof of Theorem 4.6 in \cite{fm}, which is completely general). The long exact sequence of homotopy groups gives an exact sequence \[\pi_1\diff(M)\xra{\eta_*}\pi_1(M,*)\xra{P}\pi_0\diff(M,*)\ra\pi_0\diff(M)\ra 0.\]From the abstract construction of the long exact sequence of a fibration, it is easy to see that the connecting homomorphism $P$ is equal to Push. The fact that $\ker(\text{Push})$ is central in $\pi_1(M,*)$ follows from the fact that the image of $\eta_*$ is contained in the center of $\pi_1(M,*)$ (see \cite{hatcher} Page 40). In particular, if $M$ is a nonpositively curved locally symmetric manifold of noncompact type, then the center of $\pi_1(M)$ is trivial, so Push is injective.

\subsection{Point-pushing as monodromy}\label{sec:monodromy}

In this subsection we explain the geometric nature of Problem \ref{prob}. Let $(M,*)$ and $(F,\star)$ be based manifolds. Let $X=M\ti F$, and denote by $p_M$ and $p_F$ the projections to $M$ and $F$, respectively. Consider $X$ as a bundle over $M$, and let $\si:M\ra X$ be any section with $\si(*)=\star$. Using $\si$, one can define local trivializations on $X\ra M$ so that the transition maps lie in $\diff(F,\star)$. It is not hard to see that the monodromy $\pi_1(M,*)\ra\pi_0\diff(F,\star)$ of this bundle is the composition 
\[\pi_1(M)\xra{(p_F\circ\si)_*}\pi_1(F)\xra{\text{Push}}\pi_0\diff(F,\star).\]

Now consider the special case when $(F,\star)=(M,*)$, so that $X=M\ti M$. Take $\si$ to be the diagonal map $\De :M\ra M\ti M$. Then as a bundle with section, $X$ has monodromy
\[\text{Push}: \pi_1(M)\ra\pi_0\diff(M,*).\]
This interpretation of Push allows us to interpret a lift $\vp:\pi_1(M)\ra\diff(M,*)$ in (\ref{diag:q1}) as follows. Recall that for any pair of manifolds $M,F$, a smooth $F$-bundle $E\ra M$ is determined by a map $M\ra B\diff(F)$. More precisely, there is a bijective correspondence
\[\left\{\begin{array}{cc}\text{Homotopy classes of maps}\\M\ra B\diff(F)\end{array}\right\}\longleftrightarrow \left\{\begin{array}{cc}\text{Isomorphism classes of}\\ \text{$F$-bundles } E\ra M\end{array}\right\}.\]
Similarly, a \emph{flat} $F$-bundle over $M$ is determined by a holonomy homomorphism $\pi_1(M)\ra\diff(F)$ (this is explained further in Section \ref{sec:flat}), and there is a bijective correspondence 
\[\left\{\begin{array}{cc}\text{Conjugacy classes of}\\\text{representations }\pi_1(M)\ra\diff(F)\end{array}\right\}\longleftrightarrow \left\{\begin{array}{cc}\text{Isomorphism classes of flat }\\\text{$F$-bundles } E\ra M\end{array}\right\}.\]
In a similar fashion, homotopy classes of maps $M\ra B\diff(F,\star)$ are in bijective correspondence with isomorphism classes of $F$-bundles $E\ra M$ with a \emph{distinguished section}, and  conjugacy classes of representations $\pi_1(M)\ra\diff(F,\star)$ are in bijective correspondence with isomorphism classes of $F$-bundles $E\ra M$ with a distinguished section and a foliation transverse to the fibers such that the section is one of the leaves.  

In particular, the bundle $X=M\ti M\ra M$ with section $\De:M\ra M\ti M$ defined above is classified by a map $f: M\ra B\diff(M,*)$. The monodromy is the induced map on fundamental groups
\[f_*=\text{Push}:\pi_1(M)\ra \pi_1\big(B\diff(M,*)\big)\simeq\pi_0\diff(M,*).\]
$M\ti M\ra M$ is flat (with respect to $\De$) if $M\ti M$ has a foliation transverse to the ``vertical" foliation (whose leaves are $\{x\}\ti M$) and such that the diagonal $\{(x,x): x\in M\}\sbs M\ti M$ is one of the leaves. If $M\ti M\ra M$ is flat with respect to $\De$, then the holonomy $\vp:\pi_1(M)\ra\diff(M,*)$ realizes Push.

In \cite{bcs}, Bestvina-Church-Souto show that Push is \emph{not} realized for $M=S_g$ a closed surface of genus $g\ge2$. The cases $g=0,1$ are uninteresting: For $g=0$ there is no question since $\pi_1(S_0)=0$; for $g=1$ a lift $\vp$ \emph{does} exist, but only because the map Push$:\pi_1(\ts^2)\ra\pi_0\diff(\ts^2,*)$ is trivial. This is illustrated topologically in Figure 2. For a punctured surface, the fundamental group is free, and so Push is also realized in this case.

\begin{figure}[h!] \label{fig:ppt}
\labellist 
\small\hair 2pt 
\pinlabel $f_\ga(\eta)$ at 463 166
\pinlabel $\eta$ [bl] at 255 150 
\pinlabel $\ga$ [r] at 239 106 
\endlabellist 
\centering 
\includegraphics[scale=0.5]{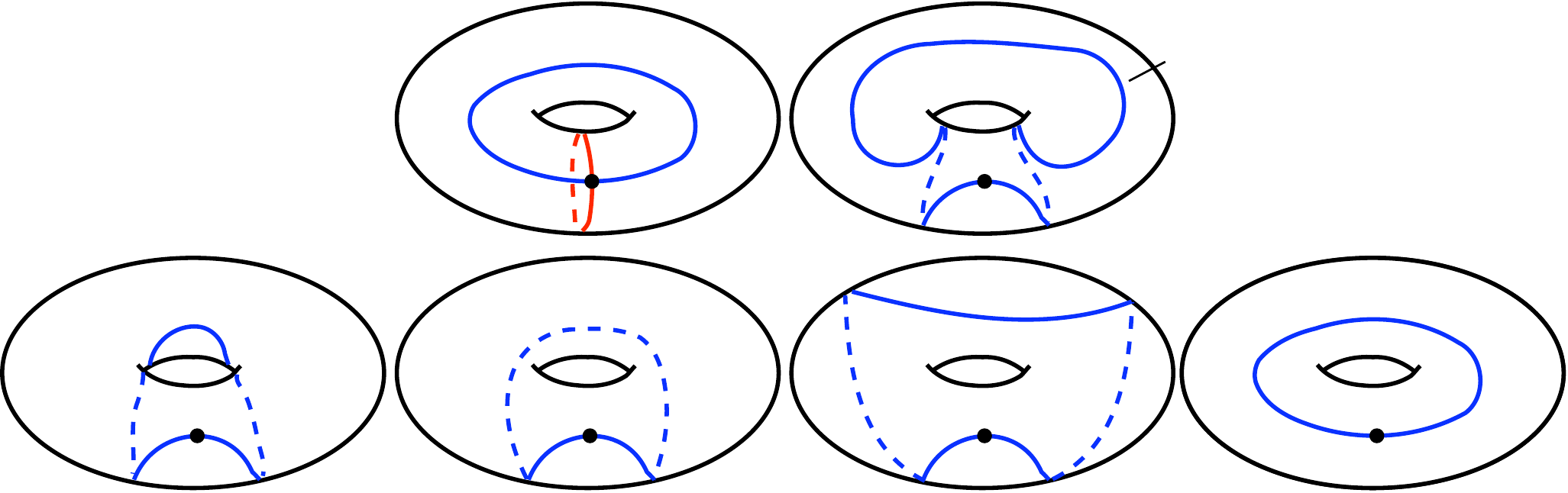} 
\caption{Shows a homotopy (rel basepoint) between $f_\ga(\eta)$ and $\eta$; from this and the fact that $f_\ga(\ga)=\ga$ it follows that Push($\ga$) is trivial in $\diff(M,*)$.}  
\end{figure}

\section{Characteristic classes and flat bundles}\label{sec:flat}

In this section we recall some well-known facts about characteristic classes and flat bundles. Then we describe a construction of the Pontryagin classes of a locally symmetric manifold using the action on the visual boundary by homeomorphisms. The material of this section will be used heavily in Sections \ref{sec:bordant} and \ref{sec:proofs}. 

\vspace{.1in}

\noindent{\bf Euler class of sphere bundles.} We briefly recall the obstruction theory definition of the Euler class, which we will use in the next section. For details see \cite{Steenrod} Section 32.  Let $\pi: E\ra B$ be an oriented topological $F$-bundle, and assume that $B$ is triangulated. Let $m$ be the smallest integer so that $\pi_m(F)$ is nontrivial. Choose a section $\si:B^{(m)}\ra E$ over the $m$-skeleton $B^{(m)}$; this can be done inductively using the fact that $\pi_i(F)=0$ for $i<m$.  The obstruction to extending the section from the $m$-skeleton to the $(m+1)$-skeleton is measured by a simplicial $(m+1)$-cochain
\[e: \{ (m+1)\text{-simplices of }B\}\ra \pi_{m}(F).\]
It is easy to show that $e$ is a cocycle and that different choices of the section on the $m$-skeleton $B^{(m)}$ define cohomologous cocycles. Then the class $[e]\in H^{m+1}\big(B;\pi_m(F)\big)$ depends only on the bundle $E\ra B$. When $F$ is a sphere, $[e]$ is called the \emph{Euler class} of $E$ and is denoted $e(E)$.

\vspace{.1in}

\noindent{\bf  Pontryagin classes of sphere bundles.} The Pontryagin classes $p_i\in H^*(BO_n)$ are invariants of real vector bundles. The following proposition shows that these invariants can also be defined for topological $\re^n$-bundles. 
\begin{prop}\label{prop:ratsurj}
 The inclusion $g:O_{n}\hra\emph{\homeo}(\re^n)$ induces a surjection  
\[g^*:H^*\big(B\emph{\homeo}(\re^n)\big)\ra H^*\big(BO_{n}\big)\]
with rational coefficients.
\end{prop}

This is indicated by Novikov's theorem on topological invariance of Pontryagin classes. The proof below uses Kirby-Siebenmann \cite{ks}, and was explained to the author via email by A.\ Hatcher. 

\begin{proof}[Proof of Proposition \ref{prop:ratsurj}]
Let $\Top_n$ be the semi-simplicial group for which $B\Top_n$ is a model for the classifying space $B\homeo(\re^n)$. There is a fibration
\[\Top_n/\text{O}_n\ra B\text{O}_n\ra B\text{Top}_n\]
In the limit $n\ra\infty$, this sequence becomes $\Top/\text{O}\ra B\text{O}\ra B\text{Top}$. Kirby-Siebenmann \cite{ks} show that Top/O has finite homotopy groups; it follows that Top/O has finite homology groups with $\z$ coefficients. Then $BO\ra B\text{Top}$ is a rational homology equivalence, and there exists a unique class $\wtil p_i\in H^{4i}(B\Top;\q)$ that restricts to the Pontryagin class $p_i\in H^{4i}(BO;\q)$. Now to obtain the proposition, note that the restriction 
\[H^*(B\text{Top}) \ra H^*(BO)\ra H^*(BO_n)\] can also be factored as $H^*(B\text{Top})\ra  H^*(B\text{Top}_n) \ra H^*(BO_n)$. Thus the restriction of $\wtil p_i$ to $H^{4i}(B\Top_n;\q)$ gives a class that maps to $p_i\in H^{4i}(BO_n;\q)$, as desired.
\end{proof}

From Proposition \ref{prop:ratsurj}, Pontryagin classes of $S^{n-1}$-bundles can be defined as follows. Define a homomorphism $\al:\homeo(S^{n-1})\ra\homeo(\re^{n})$ using the Alexander trick: $\al(f)$ performs the homeomorphism $f$ on the sphere of radius $r$ for every $r>0$, and $\al(f)$ fixes the origin. This induces maps between classifying spaces and hence a map
\[\al^*:H^*\big(B\homeo(\re^{n})\big)\ra H^*\big(B\homeo(S^{n-1})\big)\]
Note that the restriction of $\al$ to the subgroup $O_n\sbs\homeo(S^{n-1})$ is the standard action $O_n\ra\homeo(\re^n)$, so there is a commutative diagram
\begin{equation}\label{diag:ratsurj}
\begin{gathered}
\begin{xy}
(-22,0)*+{H^*\big(B\homeo(\re^n)\big)}="A";
(22,0)*+{H^*\big(B\homeo(S^{n-1})\big)}="B";
(0,-10)*+{H^*(BO_n)}="C";
{\ar"A";"B"}?*!/_3mm/{\al^*};
{\ar "A";"C"}?*!/^3mm/{g^*};
{\ar "B";"C"}?*!/_3mm/{};
\end{xy}
\end{gathered}
\end{equation}

By Proposition \ref{prop:ratsurj}, there is a class $\wtil p_i\in H^{4i}\big(B\homeo(\re^n)\big)$ with $g^*(\wtil p_i)=p_i$. Since Diagram \ref{diag:ratsurj} commutes, $\al^*(\wtil p_i)\in H^{4i}\big(B\homeo(S^{n-1})\big)$ is nontrivial. We refer to the classes $q_i=\al^*(\wtil p_i)$ as the Pontryagin classes of topological $S^{n-1}$-bundles, and we refer to the class $q=1+q_1+\cd+q_{[n/2]}$ as the \emph{total Pontryagin class}. 

\vspace{.1in}

\noindent{\bf  Flat bundles.} Let $F$ be a topological space and let $\ca G\sbs\homeo(F)$ a subgroup. Let $B$ be a manifold and let $\rho:\pi_1(B)\ra\ca G$ be a homomorphism. Define $E$ as the quotient of $\wtil B\ti F$ by the diagonal action of $\pi_1(B)$, where $\pi_1(B)$ acts by deck transformations on $\wtil B$ and by $\rho$ on $F$. Then $E$ has a natural projection $E\ra B$ with fiber $F$. An $F$-bundle $E\ra B$ obtained from this construction is called a \emph{flat $\ca G$-bundle} or a bundle with \emph{flat $\ca G$ structure}. 

A flat bundle comes equipped with a foliation: The space $\wtil B\ti F$ is naturally foliated by subspaces $\wtil B\ti\{f\}$ for $f\in F$, and this foliation descends to a foliation on $E$ so that the leaves are covering spaces of the base. The existence of such a foliation allows one to define parallel transport along curves in the base, in a way that is well-defined up to homotopy (preserving endpoints). Then parallel transport of loops at a basepoint $*\in B$ define a holonomy map $\pi_1(B,*)\ra\homeo(F)$, and this recovers the representation $\rho$ that was used to define $E$. This shows that flat bundles $\pi:E\ra B$ are characterized by the existence of a foliation on $E$ whose leaves project to $B$ as covering spaces.

Let $E\ra B$ be a flat $\ca G$-bundle. Because $E$ has parallel transport that is well defined up to homotopy, one can define local trivializations on $E\ra B$ so that the transition maps are locally constant. Then the structure group of $E$ reduces to $\ca G^\de$, which is the group $\ca G$ viewed as a topological group with the discrete topology. The classifying space for bundles with flat $\ca G$ structure is the space $B\ca G^\de$. The identity map $G^\de\ra G$ is continuous and induces a map $B\ca G^\de\ra B\ca G$ which corresponds to forgetting the flat structure. 

\vspace{.1in}

\noindent{\bf  Significance of the structure group.} Here is an example that illustrates the importance of the structure group $\ca G$ in the definition of flat $\ca G$-bundle. Let $S_g$ be a closed surface of genus $g\ge2$, and let $E\ra S_g$ be an oriented topological circle bundle. By definition, the structure group of $E\ra S_g$ is contained in the group of orientation-preserving homeomorphisms $\homeo(S^1)$. In fact, since $\homeo(S^1)$ deformation retracts to the subgroup of rotations $\SO(2)$, the bundle $E\ra S_g$ is isomorphic to a bundle $E'\ra S_g$ that has structure group $\SO(2)$. Hence the set of isomorphism classes of circle bundles over $S_g$ does not change if one changes the structure group from $\homeo(S^1)$ to $\SO(2)$ or $\Sl_2\re$ or $\psl_2\re$.   

The story is different for flat bundles. It is not hard to show that the only flat $\SO(2)$ circle bundle is the trivial bundle $E=S_g\ti S^1$. On the other hand, the unit tangent bundle $T^1S_g\ra S_g$ has a flat $\psl_2\re$ structure, and hence also a flat $\homeo(S^1)$ structure. In contrast, $T^1S_g$ does \emph{not} have a flat $\gl_2^+\re$ structure by Milnor's inequality \cite{bg} (stated in Theorem \ref{thm:bg}) because the Euler number of the unit tangent bundle is $\chi(S_g)=2-2g$.

\vspace{.1in}

\noindent{\bf  Pontryagin classes of a locally symmetric manifold.} Let $G$ be a semisimple real Lie group without compact factors and let $K\sbs G$ be a maximal compact subgroup. Let $n=\dim G/K$. The manifold $G/K$ is contractible and has a metric of nonpositive curvature so that $G$ acts on $G/K$ isometrically. In addition $G$ acts on the visual boundary $\pa (G/K)\simeq S^{n-1}$ \cite{bgs}. In general, the visual boundary of a contractible, nonpositively-curved manifold has no natural smooth structure, so even though $G/K$ is an algebraic example, the action on $\pa(G/K)$ is in general only by homeomorphisms. 

Fix $\Ga\sbs G$ a lattice and denote $M=\Ga\bs G/K$. The sequence
\[\Ga\hra G^\de\ra G\ra \homeo(S^{n-1})\] of maps of topological groups induces a sequence of maps of classifying spaces 
\begin{equation}\label{eqn:pont}M\ra BG^\de\ra BG\ra B\homeo(S^{n-1}).\end{equation} 
Under this map, the universal sphere bundle over $B\homeo(S^{n-1})$ pulls back to the unit tangent bundle of $M$. This is shown in the lemma below.
\begin{lem}\label{lem:tang}
Let $M$ be a complete Riemannian manifold of nonpositive curvature with universal cover $\wtil M$. The sphere bundle with monodromy given by the action of the deck group $\pi_1(M)$ on the ideal boundary $\pa\wtil M\simeq S^{n-1}$ is isomorphic to the unit tangent bundle of $M$. 
\end{lem}

Lemma \ref{lem:tang} is proved in \cite{tshishiku2}. Here we discuss how to use (\ref{eqn:pont}) to study the Pontryagin classes of locally symmetric manifolds. For a more complete investigation, see \cite{tshishiku2}. 

Lemma \ref{lem:tang} implies that the pullback of $q_i\in H^{4i}\big(B\homeo(S^{n-1})\big)$ along the map in (\ref{eqn:pont}) is the $i$-th Pontryagin class $p_i(TM)\in H^{4i}(M)$. Hence the kernel of the map $H^*(BG)\ra H^*(BG^\de)$ is a potential obstruction to $p_i(TM)$ being nonzero. The following theorem allows us to compute $H^*(BG)\ra H^*(BG^\de)$ for $G$ a semisimple Lie group. 

\begin{thm}[See \cite{milnor_liediscrete}]\label{thm:milnor}
Let $G$ be a real semisimple, connected Lie group; let $G_\co$ be its complexification. The sequence
\[H^*(B G_\co;\q)\ra H^*(BG;\q)\ra H^*(BG^\de;\q)\]
induced by the maps $G^\de\ra G\ra G_\co$ is ``exact" in the sense that the kernel of $H^*(BG)\ra H^*(BG^\de)$ is the ideal generated by the images of $H^i(BG_\co)\ra H^i(BG)$ for $i>0$.\end{thm}

The proof of Theorem \ref{thm:milnor} uses Chern-Weil theory. As an illustrative example, consider the case when $G$ is a real compact Lie group. Then $G$ is the maximal compact subgroup of $G_\co$, so the inclusion $G\hra G_\co$ is a homotopy equivalence. This implies that $H^*(BG_\co)\ra H^*(BG)$ is an isomorphism. Then by Theorem \ref{thm:milnor} the map $H^*(BG)\ra H^*(BG^\de)$ is identically zero. For $G=O_n$ this is the familiar fact from Chern-Weil theory that a vector bundle with a flat (that is, curvature 0) metric connection has vanishing Pontryagin classes. We now apply Theorem \ref{thm:milnor} to $G=\gl_n^+\re$. This computation will be used in the proof of Theorem \ref{thm:pont} in Section \ref{sec:proofs}.

\begin{cor}\label{cor:flatccs}
The kernel of the map 
\[j^*: H^*\big(B\emph{\gl}_n(\re);\q\big)\ra H^*\big(B\emph{\gl}_n(\re)^\de;\q\big)\]
is the algebra generated by the Pontryagin classes. 
\end{cor}
By Corollary \ref{cor:flatccs}, the only interesting characteristic class of flat $\gl_n^+\re$-bundles is the Euler class (which is nonzero only for $n$ even); moreover, since $e^2=p_{n/2}$, the square of the Euler class is \emph{not} a characteristic class of flat $\gl_n^+\re$-bundles. 
\begin{proof}[Proof of Corollary \ref{cor:flatccs}]
The complexification of $G=\gl_{n}^+\re$ is $\gl_n\co$. By Theorem \ref{thm:milnor}, to compute the kernel of $H^*(BG)\ra H^*(BG^\de)$ it suffices to compute the image of $i^*:H^*(B\gl_n\co)\ra H^*(B\gl_n^+\re)$. From a theorem of Borel \cite{borel:toplie}, $H^*(B\gl_n\co)$ is a polynomial algebra $\q[c_1,\ld,c_n]$ in the Chern classes and $H^*(B\gl_n^+\re)$ is a quotient of $\q[p_1,\ld,p_{[n/2]}, e]$ by the ideal $(e)$ for $n$ odd and $(e^2-p_{n/2})$ for $n$ even. 

Now $i^*(c_{2i})=p_i$ holds tautologically because the Pontryagin classes of a real linear bundle are defined by taking Chern classes of the complexified bundle. In addition, $i^*(c_{2i-1})=0$ because the odd Chern classes of a complexified bundle vanish; see \cite{ms} for details. Thus the image of $i^*$ is precisely the algebra generated by the Pontryagin classes. 
\end{proof}

\section{Characteristic classes of fiberwise bordant bundles}\label{sec:bordant}

This section is devoted to the study of characteristic classes of fiberwise bordant $S^{n-1}$-bundles. Here is the main definition. 

\begin{deftn}
Two $F$-bundles $E_0\ra M$ and $E_1\ra M$ are \emph{fiberwise bordant} if there exists an $F\ti[0,1]$-bundle $E\ra M$ so that for $i=0,1$, the restriction in each fiber to $F\ti\{i\}$ determines a bundle isomorphic to $E_i\ra M$. The bundle $E$ is called a \emph{(fiberwise) bordism} between $E_0$ and $E_1$. 
\end{deftn}

We will prove that fiberwise bordant $S^{n-1}$-bundles have the same Euler and Pontryagin classes. We begin with the Euler class.

\begin{lem}\label{lem:sameeuler}
If $E_0\ra M$ and $E_1\ra M$ are fiberwise bordant $S^{n-1}$-bundles, then $E_0$ and $E_1$ have the same Euler class. 
\end{lem}
\begin{proof}
Let 
\[S^{n-1}\ti[0,1]\ra E\ra M\] be a bordism between $E_0\ra M$ and $E_1\ra M$. To compare the Euler classes, we use the obstruction theory definition of the Euler class given in Section \ref{sec:flat}. Let  
\[A=\pi_{n-1}(S^{n-1}\ti 0)\simeq\pi_{n-1}(S^{n-1}\ti 1),\] let $A'=\pi_{n-1}(S^{n-1}\ti[0,1])$, and let $\al: A\xra\sim A'$ be the isomorphism induced by the inclusion of a component of the boundary. For $i=0,1$ choose sections $\si_i: M^{(n-1)}\ra E_i$ defined over the $(n-1)$-skeleton of $M$. Obstruction theory gives simplicial cocycles $e_i:C_n(M)\ra A$. Our aim is to show that $[e_0]=[e_1]$ in $H^n(M;A)$. 

The inclusions $E_i\hra E$ induce sections $M^{(n-1)}\xra{\si_i}E_i\hra E$, and the associated obstruction cocycles are
\[e_i'=\al_*(e_i): C_n(M)\xra{e_i}A\xra\al A'.\]
Now $[e_0']=[e_1']$ in $H^n(M; A')$ because the obstruction class in cohomology is independent of the section. Then since $\al_*:H^*(M;A)\xra\sim H^*(M;A')$ is an isomorphism and $\al_*[e_i]=[e_i']$, it follows that $[e_0]=[e_1]$. 
\end{proof}

Next we show that fiberwise bordant $S^{n-1}$-bundles have the same Pontryagin classes. As discussed in Section \ref{sec:flat}, the Alexander trick defines a homomorphism $\homeo(S^{n-1})\ra\homeo(\re^n)$, so a $S^{n-1}$-bundle induces an $\re^n$-bundle. Two $\re^n$-bundles are \emph{stably isomorphic} if they become isomorphic after adding a trivial bundle. More precisely, if $f_0,f_1:M\ra B\homeo(\re^n)$ classify $\re^n$-bundles, then the bundles are stably isomorphic if $f_0$ and $f_1$ become homotopic after composing with $B\homeo(\re^n)\ra B\homeo(\re^{n+k})$ for some $k\ge0$. To show fiberwise bordant $S^{n-1}$-bundles have the same Pontryagin classes we show the induced $\re^n$-bundles are stably isomorphic. 

\begin{prop}\label{prop:stabequiv}
Let $E_0\ra M$ and $E_1\ra M$ be fiberwise bordant $S^{n-1}$-bundles. Then the induced $\re^n$-bundles are stably isomorphic.
\end{prop}
Since adding a trivial bundle does not change the Pontryagin classes, the following corollary is immediate. 
\begin{cor}\label{cor:samepont}Fiberwise bordant $S^{n-1}$-bundles have the same Pontryagin classes.
\end{cor}

The proof of Proposition \ref{prop:stabequiv} given below is due to O.\ Randal-Williams and replaces a longer argument involving block bundles that appeared in an earlier version of this paper.  

\begin{proof}[Proof of Proposition \ref{prop:stabequiv}]\mbox{ }

\vspace{.1in}
\noindent\un{Step 1.} If $E_0$ and $E_1$ are fiberwise bordant $S^{n-1}$-bundles, then the induced $\re^n$-bundles are also fiberwise bordant. Let $\homeo_\pa\big(S^{n-1}\ti[-1,1]\big)$ denote the group of homeomorphisms of $S^{n-1}\ti[-1,1]$ that preserve (i.e.\ do not swap) the boundary components $S^{n-1}\ti\{-1\}$ and $S^{n-1}\ti\{1\}$. 

For $\vp\in\homeo_\pa\big(S^{n-1}\ti[-1,1]\big)$, let $\vp_{-1}, \vp_1:S^{n-1}\ra S^{n-1}$ be the induced homeomorphisms of the boundary components. By the Alexander trick $\vp_{-1},\vp_1$ can be coned off to homeomorphisms of $\bb D^n$. Coning the boundary components of $S^{n-1}\ti[-1,1]$ defines a homeomorphism of 
\[\big(\bb D^n\ti\{-1\}\big)\cup \big(S^{n-1}\ti[-1,1]\big)\cup\big(\bb D^n\ti\{1\}\big)\simeq S^n,\]
and after applying the Alexander trick once more, we obtain a homeomorphism of $\bb D^{n+1}\simeq\bb D^n\ti[-1,1]$. Ultimately this defines a homomorphism 
\[\homeo_\pa(S^{n-1}\ti[-1,1])\ra\homeo_\pa(\bb D^n\ti[-1,1])\ra\homeo_\pa(\re^n\ti[-1,1]).\]
The second homomorphism is induced by the inclusion $\re^n\hra\bb D^n$ as the interior of the closed disk. 

\vspace{.1in}
\noindent\un{Step 2.} Fiberwise bordant $\re^n$-bundles are isomorphic after adding a trivial bundle. We have two homomorphisms 
\[
\ep_{+},\ep_-:\homeo_\pa(\re^n\ti[-1,1])\rightrightarrows\homeo(\re^n)\xra{(\cdot)\ti\bbm1_\re}\homeo(\re^n\ti\re).\]
The first pair of homomorphisms is the restriction to either of the boundary components. To show that fiberwise bordant bundles are isomorphic after adding a trivial bundle it suffices to show that $\ep_+$ and $\ep_-$ are homotopic (through homomorphisms). We argue this as follows. 

Given $\vp\in\homeo_\pa(\re^n\ti[-1,1])$, let
\[\ep_0:\homeo_\pa(\re^n\ti[-1,1])\ra\homeo(\re^n\ti\re)\]
be the homomorphism that extends by $\vp_{-1}$ on $\re^n\ti(-\infty,-1]$ and by $\vp_1$ on $\re^n\ti[1,\infty)$. For any $\tau\in(-\infty,\infty)$, there is a similar homomorphism
\[\ep_\tau:\homeo_\pa(\re^n\ti[-1,1])\ra\homeo(\re^n\ti\re)\]
that performs $\vp$ in the interval $\re^n\ti [\tau,\tau+1]$ and extends to $\re^n\ti\re$ in the obvious way. Taking $\tau\ra\infty$ one obtains the homomorphism $\ep_+$, and taking $\tau\ra-\infty$ one obtains $\ep_-$. Since the homomorphisms $\ep_\tau$ are homotopic through homomorphisms, one concludes that $\ep_+$ and $\ep_-$ are homotopic. 
\end{proof}

\section{Arithmetic lattices, superrigidity, and Type (2) manifolds}\label{sec:res}

The proof of Theorem \ref{thm:main} uses superrigidity. We recall the statement here, and then describe how to determine if a manifold has Type (2a) or (2b). For more information, see \cite{witte-morris}. 
\begin{thm}[Margulis superrigidity]\label{thm:margulis}
Let $G$ be a connected, real, linear semisimple Lie group with $\re$-\emph{rank}$(G)\ge2$ and let $\Ga\sbs G$ be an irreducible lattice. Let $\psi:\Ga\ra\emph{\gl}_n(\re)$ be any homomorphism. Assume that the complexification of $G$ is simply connected, $G$ has no compact factors, and the Zariski closure of $\psi(\Ga)$ has no compact factors. Then there exists a continuous homomorphisms $\what \psi:G\ra\emph{\gl}_n(\re)$ that agrees with $\psi$ on a finite index subgroup of $\Ga$.
\end{thm}

The assumption that the complexification $G_\co$ is simply connected can be eliminated by passing to a finite cover $\what G\ra G$. Note that if $\Ga\sbs G$ is a lattice, then there is a finite index subgroup $\Lambda\sbs\Ga$ that is isomorphic to a lattice in $\what G$ (this follows from the fact that $\Ga$ is residually finite). For this reason, passing to the finite cover $\what G$ will not affect our arguments in Section \ref{sec:proofs}. 

The key condition for applying Theorem \ref{thm:margulis} to a higher rank lattice is that the Zariski closure $\ov{\psi(\Ga)}$ has no compact factors. If $\Ga$ is nonuniform, this condition holds for \emph{every} representation (as mentioned in the introduction). 

Let $M=\Ga\bs G/K$ be a Type (2a) manifold, and let $\psi:\Ga\ra\gl_n(\re)$ be a representation. It follows from the definition of Type (2a) that $\ov{\psi(\Ga)}$ has no compact factors, so every representation of $\Ga$ virtually extends. If $M$ has Type (2b), then some representations of $\Ga$ will extend to $G$ and some will not. 

\vspace{.1in} 

\noindent{\bf Type (2a) versus Type (2b).} We now will use the restriction of scalars construction to describe how to determine if a Type (2) locally symmetric manifold has Type (2a) or (2b). Let $F$ be a number field and let $H\sbs\Sl_N(\re)$ be a subgroup defined over $F$; let $S_\infty$ be the real and complex places of $F$, and let $\ca O$ be the ring of integers of $F$. Let $H_{\ca O}$ be the matrices in $H$ with entries in $\ca O$. After choosing an embedding for each place, consider the diagonal embedding $\De:F\ra \bigoplus\re\op\bigoplus\co$, which defines an embedding 
\[\De:H_{\ca O}\ra \prod_{\si\in S_\infty} H^\si.\] Then $H_{\ca O}$ is a lattice in $\prod H^\si$. In fact, for any arithmetic lattice $\Ga\sbs G$, there exists $F$ and $H$ and a surjection $\prod H^\si\ra G$ with compact kernel so that the projection of $\De(H_{\ca O})$ is $\Ga$ (up to commensurability). 

If $\Ga$ is obtained by restriction of scalars with $F$ and $H$ as above, then $\Ga$ has a map to a compact group with infinite image if and only if $H^\si$ is compact for some embedding $\si:F\ra\co$. Furthermore, if $\Ga$ maps to a compact group $\Ga\ra U$ with infinite image, then the Zariski-closure of the image is a factor of $\prod H^\si$ (see \cite{Margulis} 7.6.1 on Page 243). 

When $G$ is simple, there exists an embedding $\si$ with $H^\si$ compact if and only if $F$ is a nontrivial extension of $\q$. For explicit examples see \cite{witte-morris}.

\section{Main construction}\label{sec:genproc}

In this section we show that if Push is realized by diffeomorphisms, then $TM$ has the same Euler and Pontryagin classes as a bundle with a flat $\gl_n(\re)$ connection. 

To show this, we generalize a construction from \cite{bcs} to $\dim M>2$. If Push is realized by diffeomorphisms, then the action of $\pi_1(M)$ on $(M,*)$ can be lifted to an action on the universal cover $\wtil M$ with a global fixed point $\wtil *$. This action extends to the visual boundary $\pa\wtil M$. By blowing up $\wtil M$ at $\wtil *$, we obtain an action of $\pi_1(M)$ on $S^{n-1}\ti[0,1]$, which defines a fiberwise bordism between two $S^{n-1}$-bundles. By the work of Section \ref{sec:bordant}, these two bundles have the same Euler and Pontryagin classes. One of these $S^{n-1}$-bundles is the unit tangent bundle (induced from the action of $\pi_1(M)$ on $\pa\wtil M$). The other $S^{n-1}$-bundle has a flat $\gl_n(\re)$ connection (it is induced from the action of $\pi_1(M)$ on the tangent space of $\wtil *$). Thus $TM$ has the same Pontryagin classes as a bundle with flat $\gl_n(\re)$ connection. 

Let $M$ be a complete Riemannian manifold with nonpositive sectional curvature, and let $*\in M$ be a basepoint. Set $\Ga=\pi_1(M,*)$.  Let $p:\wtil M\ra M$ be the universal cover and choose a basepoint $\wtil*\in p^{-1}(*)$. Any diffeomorphism of $M$ can be lifted to $\wtil M$; in fact there are many lifts because any lift can be composed with a deck transformation of $\wtil M$ to get another lift. This is expressed by the following short exact sequence
\[1\ra \pi_1(M)\xra{i} \diff(\wtil M)^{\pi_1(M)}\ra \diff(M)\ra1.\]
Here $i:\pi_1(M)\ra\diff(\wtil M)$ is the action by deck transformations; the middle term denotes the normalizer of the deck group and is the group of lifts of diffeomorphisms of $M$. In general, this sequence has no section; however, when restricted to $\diff(M,*)\sbs\diff(M)$ the sequence splits: define a section 
\[\si:\diff(M,*)\ra\diff(\wtil M)^{\pi_1(M)}\] by choosing $\si(f)$ to be the unique lift that fixes $\wtil *$ (choose any lift and post-compose with the appropriate deck transformation). Now suppose, for a contradiction, that there exists a lift $\vp:\Ga\ra\diff(M,*)$ of Push. Composing with $\si$ gives a homomorphism
\[\si\circ\vp: \Ga\ra\diff(\wtil M,\wtil *)^{\pi_1(M)}\]
and hence an action of $\Ga$ on $\wtil M$ with a global fixed-point. This action induces two more actions:
\begin{enumerate}
\item[(i)] $\rho_0:\Ga\ra\homeo(S^{n-1})$ is the action on the unit tangent space at the fixed point $T^1_{\wtil *}\wtil M\simeq S^{n-1}$, 
\item[(ii)] $\rho_1:\Ga\ra\homeo(S^{n-1})$ is the action on the visual boundary $\pa \wtil M\simeq S^{n-1}$ as described in Lemma \ref{lem:boundary}.
\end{enumerate}

\begin{lem}\label{lem:boundary}
Let $\ga\in\pi_1(M)$ and let $\text{\emph{Push}}(\ga)\in\pi_0\diff(M,*)$. For any diffeomorphisms $f\in\emph{\diff}(M,*)$ representing $\text{\emph{Push}}(\ga)$, the lifted diffeomorphism $\si(f)$ extends to the boundary $\pa\wtil M$ and acts on the boundary as the deck transformation $i(\ga)$.  
\end{lem}

\noindent{\it Remark.} For $M$ a closed surface, Nielsen constructed a homomorphism 
\[\rho:\pi_0\diff(M,*)\ra\homeo(S^{1}),\] 
and the restriction of $\rho$ to the point-pushing subgroup is the representation $\rho_1$ above. This is explained in \cite{fm} in Sections 8 and 5.5.4. 

\begin{proof}[Proof of Lemma \ref{lem:boundary}]
Choose an isotopy $f_t$ from the identity to $f$. This isotopy can be lifted to $\wtil M$ to an isotopy from the identity of $\wtil M$ to a map $\wtil f$ covering $f$; along this isotopy, the basepoint $\wtil *$ is moved to $i(\ga)^{-1}(\wtil *)$ (action of the deck group). In other words, $\wtil f(\wtil *)=i(\ga)^{-1}(\wtil *)$. Recall that $\si(f)$ is defined as the unique lift of $f$ that fixes $\wtil *$; therefore, \[\si(f)=i(\ga)\circ\wtil f.\] Note that $\wtil f$ moves points a uniformly bounded amount, and so $\wtil f$ extends to $\wtil M\cup\>\pa\wtil M$ and acts trivially on $\pa\wtil M$. Hence $\si(f)=i(\ga)\circ\wtil f$ extends to $\pa\wtil M$ and acts on $\pa\wtil M$ as $i(\ga)$.
\end{proof}

Our goal is to understand the relationship between the two actions $\rho_0$ and $\rho_1$. We will do this using bundle theory. Since $\Ga\simeq\pi_1(M)$, the homomorphisms $\rho_0$ and $\rho_1$ induce $S^{n-1}$-bundles with flat $\homeo(S^{n-1})$ connections
\begin{equation}\label{eqn:bundles} E_0\ra M\>\>\>\>\>\text{ and }\>\>\>\>\> E_1\ra M.\end{equation}
Furthermore, the two bundles $E_0,E_1$ are fiberwise bordant. To show this we use the following lemma. 

\begin{lem}\label{lem:compactify}
Let $\bb D^n$ denote the closed disk. Let $\Lambda\sbs\emph{\homeo}(\bb D^n)$ be the subgroup of homeomorphisms that fix the origin $0\in\bb D^n$ and are differentiable at $0$. Then there exists an action of $\Lambda$ on $[0,1]\ti S^{n-1}$ such that the restriction to $\{1\}\ti S^{n-1}$ is the $\Lambda$-action on $\pa\bb D^n$ and the restriction to $\{0\}\ti S^{n-1}$ is the $\Lambda$-action on $T_0^1\bb D^n$. 
\end{lem}

The proof of Lemma \ref{lem:compactify} uses the standard blow up construction. We prove the lemma at the end of this section. Let us explain why Lemma \ref{lem:compactify} implies that $E_0\ra M$ and $E_1\ra M$ are fiberwise bordant. Note that the compactification $X=\wtil M\cup\pa\wtil M$ is diffeomorphic to the closed disk $\bb D^n$ because $M$ has nonpositive curvature. Assuming Push is realized, the induced action of $\Ga$ on $\wtil M$ gives an action on $\bb D^n$ that is smooth on the interior and has global fixed point 0. This defines a homomorphism $\Ga\ra\Lambda\sbs\homeo(\bb D^n)$, so by Lemma \ref{lem:compactify}, $\Ga$ acts on $S^{n-1}\ti[0,1]$ so that the restriction to $S^{n-1}\ti\{0\}$ is the $\Ga$-action on $T_{\wtil *}^1\wtil M$ and the restriction to $S^{n-1}\ti\{1\}$ is the $\Ga$-action on $\pa\wtil M$. Define $E\ra M$ to be the $S^{n-1}\ti[0,1]$-bundle induced by the action of $\Ga$ on $S^{n-1}\ti[0,1]$. Then $E$ is a bordism between $E_0\ra M$ and $E_1\ra M$. 

Let $G=\isom(\wtil M)$. Note that the structure group of $E_1$---the bundle whose monodromy is the $\Ga$-action on $\pa\wtil M$---is contained in $G\sbs\homeo(S^{n-1})$ because $\Ga$ acts as the deck group on $\pa\wtil M$ and the deck group action on $\pa\wtil M$ extends to $G$. Similarly, the structure group of $E_0$---the bundle induced by the action of $\Ga$ on $T^1_{\wtil *}\wtil M$---is contained in the image of $\gl_n^+(\re)\ra\homeo(S^{n-1})$ because $\Ga$ acts on $T_{\wtil *}\wtil M$ linearly. In other words, the representations $\rho_0$ and $\rho_1$ factor:
\begin{equation*}\label{eqn:rho0}
\rho_0:\Ga\ra\gl_n^+(\re)^\de\ra\gl_n^+(\re)\ra\homeo(S^{n-1})\end{equation*}
and 
\begin{equation*}\label{eqn:rho1}\rho_1:\Ga\ra G^\de\ra G\ra \homeo(S^{n-1}).\end{equation*}
These maps produce Diagram \ref{diag:maincoho}. By Lemma \ref{lem:sameeuler} and Corollary \ref{cor:samepont}, $E_0\ra M$ and $E_1\ra M$ have the same Euler and Pontryagin classes, and so Diagram \ref{diag:maincoho} commutes on the Euler and Pontryagin classes in $H^*\big(B\homeo(S^{n-1})\big)$. Then to show that Push is not realized, it suffices to show that Diagram \ref{diag:maincoho} does not commute on cohomology. This is done in Section \ref{sec:proofs} for manifolds of Type (1) and (2a). 

\begin{proof}[Proof of Lemma \ref{lem:compactify}]
Give $[0,1]\ti S^{n-1}$ coordinates $(t,\ta)$. Viewing a point $\ta\in S^{n-1}$ as a unit vector in $\re^n$, define 
\[\begin{array}{lcll}
\pi:&[0,1]\ti S^{n-1}&\ra&\bb D^n\\
&(t,\ta)&\mapsto &t\ta.\end{array}\]
When restricted to $(0,1]\ti S^{n-1}$ this map is a diffeomorphism onto its image $\bb D^n\bs\{0\}$. Identify $\{0\}\ti S^{n-1}$ with the space of rays through the origin in $T_0\bb D^n$. Note that $\Lambda$ acts on this space because $\Lambda$ acts differentiably at 0. Now for $f\in \Lambda$, define $\wtil f:S^{n-1}\ti[0,1]\ra S^{n-1}\ti[0,1]$ by 
\[\wtil f(t,\ta)=\left\{\begin{array}{cll}
\pi^{-1}\circ f\circ\pi(t,\ta)&t>0\\[2mm]
(0,df_0(\ta))&t=0\end{array}\right.\]
It is an easy exercise to show that $\wtil f$ is a homeomorphism. It is obvious that $\wtil f$ restricts to $\{0,1\}\ti S^{n-1}$ as desired. Then $f\mapsto\wtil f$ defines the desired homomorphism $\Lambda\ra\homeo\big([0,1]\ti S^{n-1}\big)$. 
\end{proof}

\section{Proof of the main results}\label{sec:proofs}

In this section we give proofs of Theorems \ref{thm:pont}, \ref{thm:main}, and \ref{thm:surfprod}. We begin with the easiest of our results. 

\subsection{Proof of Theorem \ref{thm:surfprod}} \label{sec:ppsurf}

In this subsection, let $M$ be a product of hyperbolic surfaces, or more generally any compact quotient of $\hy^2\ti\cd\ti\hy^2$. We will use the following generalization of Milnor's inequality to show that $\text{Push}:\Ga\ra \pi_0\diff(M,*)$ is not realized by diffeomorphisms.

\begin{thm}[Bucher-Gelander \cite{bg}]\label{thm:bg}
Let $M$ be a compact quotient of $\hy^2\ti\cd\ti\hy^2$ $(k$ times$)$. Let $E\ra M$ be a flat $\emph{\gl}_{2k}^+(\re)$-bundle over $M$. Then 
\[|\eu(E)|\le \fr{1}{2^k}|\eu(TM)|.\]
\end{thm}
The case $k=1$ is due to Milnor. 

\begin{proof}[Proof of Theorem \ref{thm:surfprod}]
By Theorem \ref{thm:bg}, the unit tangent bundle of $M$ has no flat $\gl_{2k}^+(\re)$ connection. Furthermore, since fiberwise bordant bundles have the same Euler class (by Lemma \ref{lem:sameeuler}), Theorem \ref{thm:bg} implies that the unit tangent bundle cannot be fiberwise bordant to a bundle with a flat $\gl_{2k}^+(\re)$ connection. However, if Push is realized, then by the construction of Section \ref{sec:genproc}, there exists a fiberwise bordism between the unit tangent bundle $T^1M$ and a bundle with flat $\gl_{2k}^+(\re)$ connection. This contradiction shows that Push is not realized by diffeomorphisms.
\end{proof}

The argument above is essentially the argument given in \cite{bcs}; that paper focuses on the case $M$ is a surface, so $k=1$. For $k=1$, Theorem \ref{thm:bg} is known as Milnor's inequality (see \cite{bg}). 

One might wonder if the proof of Theorem \ref{thm:surfprod} extends to other locally symmetric manifolds. Unfortunately, with the current known results on bounded cohomology, the answer is no. For example, the known bounds on the Euler number of a flat $\gl_{2k}(\re)$-bundle over a hyperbolic $2k$-manifold are due to Smillie (see the comment on Page 3 in \cite{bg}): If $M^{2k}$ is hyperbolic and if $E\ra M$ is a flat $\gl_{2k}^+(\re)$-bundle, then 
\[|\eu(E)|\le \fr{\pi^k}{2^{k}\cdot(2k-1)!!\cdot v_{2k}}|\eu(TM)|.\]
Here $(2k-1)!!=\prod_{i=1}^k(2i-1)$ and $v_{2k}$ is the volume of a regular ideal $(2k)$-simplex in $\hy^{2k}$. Unfortunately, for $k\ge2$ the fraction on the right-hand-side is greater than 1, and so does \emph{not} provide an obstruction to $TM$ being bordant to a bundle with a flat $\gl_{2k}^+(\re)$ connection. In the next section we exhibit the Pontryagin classes as obstructions to point-pushing on nonpositively curved manifolds. 

\subsection{Proof of Theorem \ref{thm:pont}}\label{sec:pont} 

\begin{proof}[Proof of Theorem \ref{thm:pont}]
Suppose for a contradiction that Push is realized. In Section \ref{sec:genproc}, we showed that this produces Diagram (\ref{diag:maincoho}) that commutes on the Pontryagin classes and Euler class in $H^*\big(\homeo(S^{n-1})\big)$. Let $q\in H^*\big(B\homeo(S^{n-1})\big)$ be the total Pontryagin class defined in Section \ref{sec:flat}. Using the notation of Diagram \ref{diag:maincoho}, the pullback $b^*j^*d^*(q)$ is trivial since 
\[d^*(q)=1+p_1+\cd+p_{[n/2]}\] is the total Pontryagin class in $H^*(B\gl_n(\re))\simeq H^*(BO_n)$, and 
\[j^* (1+p_1+\cd+p_{[n/2]})=1\] by Corollary \ref{cor:flatccs}. On the other hand, $a^*i^*c^*(q)$ is $p(TM)$ by Lemma \ref{lem:tang}, and this is nontrivial by assumption. Since Diagram (\ref{diag:maincoho}) commutes, this is a contradiction. Then Push is not realized by diffeomorphisms.
\end{proof}

\begin{rmk} We give two remarks in passing. 
\begin{enumerate}
\item[(i)] In Theorem \ref{thm:pont} we actually need to assume $M$ has nontrivial Pontryagin classes, rather than just nontrivial Pontryagin numbers. For example, if $M$ is complex hyperbolic, then the dual symmetric manifold is $\co P^n$, which has zero Pontryagin numbers for $n$ odd (see for example \cite{ms}). 
\item[(ii)] There are non-locally symmetric manifolds to which Theorem \ref{thm:pont} applies. Ontaneda's Riemannian hyperbolization (building on work of Charney-Davis and Davis-Januszkiewicz) gives examples of many negatively curved Riemannian manifolds that are not locally symmetric manifolds (see \cite{on}). Specifically, for any closed manifold $N^n$, one can construct a negatively curved manifold $M^n$ together with a map $M\ra N$ so that the Pontryagin classes of $N$ pullback to the Pontryagin classes of $M$. Hence if $N$ is any manifold with nonzero Pontryagin classes, and $M$ is a hyperbolization, then $\text{Push}:\pi_1(M,*)\ra \pi_0\diff(M,*)$ is not realized by diffeomorphisms.
\end{enumerate}\end{rmk}

\subsection{Proof of Theorem \ref{thm:main}}\label{sec:higherrank} 
In this section $G$ will denote a semisimple Lie group with no compact factors and real rank at least $2$. Fix $\Ga\le G$ a lattice. Set $M=\Ga\bs G/K$ and $n=\dim M$. So far we have seen that when $M$ has nonpositive curvature, the Euler or Pontryagin classes are obstructions to realizing Push. Unfortunately, locally symmetric manifolds of Type (2) have trivial total Pontryagin class $p(TM)$, so the approaches of Sections \ref{sec:ppsurf} and \ref{sec:pont} will not work for these examples. 

To overcome this problem, recall from Section \ref{sec:genproc} that if Push is realized, we obtain a representation $\rho_0:\Ga\ra\gl_n(\re)$. If $M$ is Type (2a), then $\rho_0$ is not pre-compact, and $\rho_0$ extends to $G\ra\gl_n(\re)$ by Margulis superrigidity. We leverage this fact to reduce the realization problem to the representation theory of $G$. 

The proof of Theorem \ref{thm:main} will proceed by the steps outlined in the introduction. We complete Step (ii) in Sections \ref{sec:kact} and  \ref{sec:ccreps}. The proof of Step (iii) will be carried out in Section \ref{sec:isotropyextend}. Finally, we prove Theorem \ref{thm:main} in Section \ref{sec:prove2a}.

\subsubsection{{\bf The action of $K$ on $\pa(G/K)$}}\label{sec:kact}Here we give an algebraic description of the action of $K$ on the visual boundary $\pa(G/K)$. 

\begin{lem}
The action of $K\sbs G$ on $\pa G/K$ is induced by a linear representation $\iota: K\ra\emph{\gl}_n(\re)$. 
\end{lem}

\begin{proof} Since $G$ has noncompact type, the symmetric manifold $G/K$ has nonpositive curvature, and the visual boundary $\pa(G/K)$ can be defined as equivalence classes of geodesic rays \cite{bgs}. The exponential map defines a homeomorphism 
\begin{equation}\label{eqn:adjointboundary}s: T_{eK}^1(G/K)\ra \pa(G/K)\end{equation}
and the action of $K$ on $G/K$ induces $K$-actions on $T_{eK}^1(G/K)$ and on $\pa(G/K)$. It is easy to see that $s$ is equivariant with respect to these actions. 

The action of $K$ on $T^1_{eK}(G/K)$ can be described as follows. The adjoint action of $K$ on $\mf g=T_e(G)$ decomposes into invariant subspaces $\mf k\op\mf p$, where $\mf k=\text{Lie}(K)$ and $\mf p\simeq T_{eK}(G/K)$. Since the conjugation action and the left action of $K$ on $G$ descend to the same action on $G/K$, the action of $K$ on $T_{eK}(G/K)\simeq\mf p$ is isomorphic to the adjoint action of  $K$ on $\mf p\sbs\mf k\op\mf p=\mf g$. Thus the action of $K$ on $\pa(G/K)$ is isomorphic to the action induced by $\iota: K\ra\aut(\mf p)$.\end{proof}

\noindent We will refer to the representation $\iota:K\ra\aut(\mf p)$ as the \emph{isotropy representation}. 

\subsubsection{{\bf Characteristic classes of representations}}\label{sec:ccreps} In this section we show that the isomorphism class of a representation can be detected by the characteristic classes of that representation. Let $K$ be a compact group with maximal torus $S$, and let $\al:K\ra\gl_n(\co)$ be a continuous representation. Up to conjugation, the restriction of $\al$ to $S$ is diagonal, and there are continuous homomorphisms $\la_i:S\ra\co^\ti$, so that for every $s\in S$
\[\al(s)=\left(\begin{array}{ccccccc}
\la_1(s)&&\\
&\ddots&\\
&&\la_n(s)
\end{array}\right)\]
The $\la_i$ are called the \emph{weights} of the representation, and they uniquely determine the representation (see \cite{fh} Page 375). The space of weights is $\Hom(S,\co^\ti)\simeq H^1(S;\z)$.

The representation $\al:K\ra\gl_N(\co)$ induces a map on classifying spaces 
\[\al^*:H^*(B\gl_n(\co))\ra H^*(BK),\]
and the images $\al^*(c_i)$ of the Chern classes are invariants of the representation. In other words, conjugate representations have the same Chern classes. In fact, the converse is also true. 

\begin{prop}\label{prop:conjreps}
Let $\al,\be:K\ra\emph{\gl}_N(\co)$ be two representations. If the induced maps $\al^*,\be^*:H^*(B\emph{\gl}_n(\co))\ra H^*(BK)$ are equal, then $\al$ and $\be$ are isomorphic representations. 
\end{prop}
In short, the proof is as follows. The Chern classes of a representation can be computed by the weights of the representation, and if $\al$ and $\be$ have the same Chern classes, then they must also have the same weights. A representation is determined up to conjugacy by its weights, so representations with the same Chern classes must be isomorphic. 
\begin{proof}[Proof of Proposition \ref{prop:conjreps}]
Borel-Hirzebruch \cite{bh} give an algorithm for computing $\al^*$. Choose a maximal torus $S\sbs K$ on which $\al$ is diagonal, and let $\la_i\in  H^1(S;\z)$ be the weights as above. The transgression for the fiber sequence $S\ra ES\ra BS$ defines an isomorphism $\tau:H^1(S;\z)\ra H^2(BS;\z)$, and we set $\om_i=\tau(\la_i)$. The polynomial 
\[c(\al):=\prod_{i=1}^N\big(1+\om_i)\in H^*(BS)\]
is invariant under the action of the Weyl group of $S\sbs K$, and hence is in the image of $H^*(BK)\ra H^*(BS)$. Then according to Borel-Hirzebruch, $\al^*(c_i)$ is equal to the degree-$i$ term of $c(\al)$. 

Now if $\al^*=\be^*$, then $c(\al)=c(\be)$. Since $H^*(BS)$ is a polynomial algebra, it is a unique factorization domain, and hence the set of $\om_i$'s for $\al$ coincide with the set of $\om_i$'s for $\be$. Since the transgression $\tau$ is an isomorphism, this means that the $\la_i$'s for $\al$ coincide with the $\la_i$'s for $\be$. In other words, $\al$ and $\be$ have the same weights. Since a representation is uniquely determined by its weights (\cite{fh} Page 375), $\al$ and $\be$ are isomorphic representations. 
\end{proof}

\subsubsection{{\bf Extending the isotropy representation}} As discussed above, the proof of Theorem \ref{thm:main} we will reduce to showing the following representation theory fact. 

\begin{thm}\label{thm:noextend}
Let $M=\Ga\bs G/K$ be a Type $(2a)$ locally symmetric manifold. Then the isotropy representation $\iota: K\ra\aut(\mf p)$ does not extend to a representation of $G$.
\end{thm}

\noindent The proof is a somewhat lengthy detour and will be performed in Section \ref{sec:isotropyextend}.

\subsubsection{{\bf Conclusion}}\label{sec:prove2a}

\begin{proof}[Proof of Theorem \ref{thm:main}]
As in the proofs of Theorems \ref{thm:pont} and \ref{thm:surfprod}, we proceed by contradiction. Let $n=\dim(M)$. Let $\rho_0,\rho_1:\Ga\ra\homeo(S^{n-1})$ be the representations constructed in Section \ref{sec:genproc}, where $\rho_1$ factors through the action of $G$ on the boundary of $G/K$ and $\rho_0$ factors through some linear action $\gl_n\re\car T^1_{\wtil *}(G/K)$. This is express by the following (not-necessarily commutative) diagram
\[\begin{xy}
(40,-10)*+{\homeo(S^{n-1})}="E";
(0,-10)*+{\Ga}="B";
(20,-20)*+{\gl_n\re}="C";
(20,0)*+{G}="D";
(13,-13)*+{}="Z";
{\ar"C";"E"}?*!/^3mm/{};
{\ar "B";"C"}?*!/^3mm/{h};
{\ar "B";"D"}?*!/_3mm/{};
{\ar "D";"E"}?*!/_3mm/{};
{\ar@{-->} "D";"C"}?*!/^3mm/{};
{\ar@(ur,ul) "Z";"Z"}?*!/^3mm/{};
\end{xy}\] 
Since $M$ has Type (2a), $h$ either has finite image or its Zariski closure has no compact factors. After passing to a finite index subgroup of $\Ga$ and possibly replacing $G$ by a finite cover (see Section \ref{sec:res}), $h$ extends to $G$ by Theorem \ref{thm:margulis}. This is expressed by the dashed arrow in the diagram above and the fact that the left triangle commutes. 

In addition, this diagram commutes on the level of $H^*(B(-))$ by Lemma \ref{lem:sameeuler} and Proposition \ref{cor:samepont}. Switching focus from $\Ga$ to $K$, there is a diagram
\[\begin{xy}
(-30,0)*+{K}="A";
(-10,0)*+{G}="B";
(20,-6)*+{\gl_n\re}="C";
(20,6)*+{\homeo(S^{n-1})}="D";
{\ar@{^{(}->}"A";"B"}?*!/^3mm/{};
{\ar "B";"C"}?*!/^3mm/{\psi};
{\ar "B";"D"}?*!/_3mm/{f};
{\ar "C";"D"}?*!/^3mm/{g};
\end{xy}\]
This diagram commutes on the level of $H^*(B(-))$ because the preceding one does. We claim that this implies that $f$ and $g\circ\psi$ are isomorphic when restricted to $K$. First note that $g\circ\psi\rest{}{K}$ and $f\rest{}{K}$ are induced by \emph{linear} representations of $K$: Tautologically, $g\circ\psi$ is induced by the linear representation $\psi\rest{}{K}:K\ra\gl_n\re$. Also $f\rest{}{K}$ is induced by the isotropy representation 
\[\iota: K\ra\gl(\mf p)\simeq\gl_n\re,\] as explained in Section \ref{sec:kact}. Since $g\circ\psi$ and $f$ induce the same map on cohomology of classifying spaces, this means $\psi^*$ and $\iota^*:H^*(B\gl_n\re)\ra H^*(BK)$ are equal. By Proposition \ref{prop:conjreps}, $\psi$ and $\iota$ are isomorphic representations. 

Since $\psi$ is the restriction of a representation of $G$, the same must be true of $\iota$. However,  $\iota$ is \emph{not} the restriction of any $G$-representation by Theorem \ref{thm:noextend}. This contradiction shows that Push is not realized by diffeomorphisms.
\end{proof}

\section{Extending the isotropy representation}\label{sec:isotropyextend}

The goal of this section is to prove Theorem \ref{thm:noextend}. It is enough to prove Theorem \ref{thm:noextend} for $G$ simple. Let $G_0$ be a real simple Lie group with maximal compact subgroup $K_0\sbs G_0$. Let $\mf k_0\sbs\mf g_0$ denote the corresponding Lie algebras. The adjoint action of $K_0$ on $\mf g_0$ decomposes into invariant subspaces $\mf g_0=\mf k_0\op\mf p_0$. We want to show that the isotropy representation $\iota:K_0\ra\aut(\mf p_0)$ does not extend to a representation of $G_0$. 

To solve this problem, we convert to the Lie algebra and complexify $\mf k=\mf k_0\ot\co$, $\mf g=\mf g_0\ot\co$, and $\mf p=\mf p_0\ot\co$. We face  the following problem.

\begin{prob}\label{q:extend}
Show there is no representation of $\rho:\mf g\ra\End(V)$ whose restriction to $\mf k\sbs\mf g$ is isomorphic to the isotropy representation $\iota$.
\end{prob}

We will employ two arguments to solve this problem. One argument will apply to complex $G$ (see Section \ref{sec:extendcplx}). The other argument will apply to $G=\Sl_n(\re)$, $\SU^*_{2n}$, $\SO_{n,1}$, and $E_{6(-26)}$. 

In the latter case, we proceed by contradiction. Suppose $\rho:\mf g\ra\End(V)$ exists. Then $\rho$ must be irreducible because the isotropy representation is irreducible. Here we use that $\mf g$ is simple (see \cite{helgason2} Ch.\ VIII, Sec.\ 5). Let $\mf h_1\sbs\mf k$ and $\mf h\sbs\mf g$ be Cartan subalgebras with $\mf h_1\sbs\mf h$. Let $r:\mf h^*\ra\mf h_1^*$ be the map the restricts a weight of $\mf h$ to a weight of $\mf h_1$. Let $\la$ denote the highest weight of $V$ and let $\la_1$ denote the highest weight of $\mf p$. Since $V$ solves the extension problem, $r(\la)=\la_1$. This restricts the possible $\la$ to an affine subspace $A$ of $\mf h^*$. We further reduce the set of possible $\la$ by intersecting $A$ with the cone of dominant weights. For the remaining $\la$ we will see that $V_\la$---the irreducible representation with highest weight $\la$---does not have the right dimension (we must have $\dim V=\dim\mf p$). From this we conclude that the representation $V$ does not exist. 

\subsection{Extending the isotropy representation for $G=\Sl_n(\re)$} Here $K=\SO_n$. After complexifying we reduce to proving the following lemma. 

\begin{lem}
\label{lem:noextendSO}
Fix $n\ge2$. Let $\mf g=\mf{sl}_n(\co)$ and let $\mf k=\mf{so}_n(\co)$. Let $\mf g=\mf k\op\mf p$ be the decomposition of the adjoint representation of $\mf k$ on $\mf g$. Then $\iota:\mf k\ra\emph{\End}(\mf p)$ does not extend to a representation $\mf g\ra\emph{\End}(\mf p)$. 
\end{lem}

\begin{proof}
Let $k=\lfloor \fr{n}{2}\rfloor$. If $n$ is even (resp.\ odd) define $\mf{so}_n(\co)$ using the bilinear form  $B_k=\left(\begin{array}{cc}0&I_k\\I_k&0\end{array}\right)$ (resp.\ $\left(\begin{array}{cc}B_k&\\&1\end{array}\right)$). The diagonal matrices form a Cartan subalgebra $\mf h\sbs\mf{sl}_n(\co)$, and the restriction $\mf h_1:=\mf h\cap \mf{so}_n(\co)$ is a Cartan subalgebra of $\mf{so}_n(\co)$. Here $\mf h_1$ consists of diagonal matrices of the form $e_i-e_{i+k}$ for $i=1,\ld,k$. Define $L_1,\ld, L_n:\mf h\ra\co$ by 
\[\left(\begin{array}{ccc}x_1&&\\&\ddots&\\&&x_n\end{array}\right)\mapsto x_i.\]
The weight space $\mf h^*$ is the quotient of the free $\co$-vector space $\lan L_1,\ld,L_n\ran$ by the relation $L_1+\cd+L_n=0$. Upon restriction $r:\mf h^*\ra\mf h_1^*$, there are further relations: $r(L_i+L_{i+k})=0$ for $i=1,\ld,k$ and when $n$ is odd $r(L_n)=0$. Then $\mf h_1$ has a basis $L_1',\ld,L_k'$, where $L_i'=r(L_i)$. With this notation, it is elementary to show that the weights of the isotropy representation $\mf{so}_n(\co)\ra\End(\mf p)$ are 
\[\pm 2L_i', \>\>\>L_i'-L_j', \>\>\>\pm (L_i'+L_j'),\>\>\> 0\]
for $i\neq j$. The multiplicity of 0 is $k-1$, and all other weights have multiplicity 1. Without loss of generality, the highest weight is $2L_1'$ (this follows the convention in \cite{fh}). If this representation is the restriction of a (necessarily irreducible) representation $\mf{sl}_n(\co)\ra\End(V)$ with highest weight $\la$, then $r(\la)=2L_1'$. Then $\la$ has the form $\la=2L_1+u$ where $u$ is an integral element of $\ker r$. 

\vspace{.1in}

\noindent{\bf Case 1.} If $n$ is even, then 
\[\begin{array}{lll}
\la&=2\>L_1+a_1(L_1+L_{k+1})+\cd+a_k(L_k+L_{2k})\\[2mm]
&=2L_1+\sum_{i=1}^{k-1}a_i(L_i+L_{k+i})+a_k\>L_k-\sum_{i=1}^{2k-1}a_k\>L_i\\[2mm]
&=\big(2+a_1-a_k\big)L_1+\sum_{i=2}^k(a_i-a_k)L_i+\sum_{i=1}^{k-1}(a_i-a_k)L_{k+i}
\end{array}\]
In the second line above we have used the relation $\sum_{i=1}^n L_i=0$. Since $\la$ is a nonnegative sum of fundamental weights, the coefficient on $L_i$ is at least the coefficient on $L_{i+1}$. Then 
\[2+a_1-a_k\ge a_2-a_k\ge \cd\ge a_{k-1}-a_k\ge0\ge a_1-a_k\ge\cd\ge a_{k-1}-a_k.\]
In particular this implies that $a_i-a_k=0$ for $i=1,\ld,k-1$. Then in fact, 
\[\lambda=2\>L_1.\]
However, the representation of $\mf{sl}_n(\co)$ with highest weight $2L_1$ has dimension $\fr{n(n+1)}{2}$, which is equal to $\dim\mf p=\fr{(n-1)(n+2)}{2}$ for no values of $n$. This shows the extension does not exist when $n$ is even. 

\vspace{.1in}

\noindent{\bf Case 2.} If $n$ is odd, then 
\[\begin{array}{lll}
\la &= 2L_1+a_1(L_1+L_{k+1})+\cd+a_k(L_k+L_{2k})+a_{k+1}L_{2k+1}\\[2mm]
&=2L_1+\sum_{i=1}^ka_i\>L_i+\sum_{i=1}^ka_i\>L_{k+i}-\sum_{i=1}^{2k} a_{k+1}\>L_i\\[2mm]
&=(2+a_1-a_{k+1})L_1+\sum_{i=2}^k(a_i-a_{k+1})L_i+\sum_{i=1}^k(a_i-a_{k+1})L_{k+i}
\end{array}\]
Again, in the second line we used the relation $\sum_{i=1}^n L_i=0$. Similar to Case (1), this implies that 
\[2+a_1-a_{k+1}\ge a_2-a_{k+1}\ge \cd\ge a_k-a_{k+1}\ge a_1-a_{k+1}\ge\cd\ge a_k-a_{k+1}\ge0.\]
It follows that $a_1=a_2=\cd=a_k$ and $a_1\ge a_{k+1}$. Let $c=a_1-a_{k+1}$. Then 
\[\la=2L_1+c(L_1+\cd+L_{2k})\]
where $c\ge0$. Let $V$ be the the irreducible representation of $\mf{sl}_n(\co)$ with highest weight $\la=2L_1+c(L_1+\cd+L_{2k})$. We show $\dim V>\dim\mf p$. 
Note that 
\[\la'=\la+(L_2-L_1)=(c+1)L_1+(c+1)L_2+c(L_3+\cd+L_{2k})\]
is another weight of $V$. The Weyl group $\ca W\simeq S_n$ of $\mf h\sbs\mf{sl}_n(\co)$ acts on the weights, permuting the indices. The orbits of $\la$ and $\la'$ under $\ca W$ have size
\[|\ca W.\la|=\fr{n!}{(2k-1)!}=n(n-1)\]
and
\[|\ca W.\la'|=
n(n-1)(n-2)/2.\]
Since $\la$ and $\la'$ are in distinct $\ca W$ orbits, $|\ca W.\la|+|\ca W.\la'|$ is a lower bound on the dimension of $V$, and one can check that this lower bound is greater than the dimension of $\mf p$:
\[\dim V\ge |\ca W.\la|+|\ca W.\la'|>\dim\mf p.\]
This shows the extension $V$ cannot  exist. 
\end{proof}

\subsection{Extending the isotropy representation for $G=\SU^*_{2n}$} Here $K=\SP_n$. After complexifying we reduce to proving the following lemma.
\begin{lem}
\label{lem:noextendSP}
Fix $n\ge2$. Let $\mf g=\mf{sl}_{2n}(\co)$ and let $\mf k=\mf{sp}_{2n}(\co)$. Let $\mf g=\mf k\op\mf p$ be the decomposition of the adjoint representation of $\mf k$ on $\mf g$. Then $\iota:\mf k\ra\emph{\End}(\mf p)$ does not extend to a representation $\mf g\ra\emph{\End}(\mf p)$. 
\end{lem}

\begin{proof}
Define $\mf{sp}_{2n}(\co)$ using the bilinear form $J_n=\left(\begin{array}{cc}0&I_n\\-I_n&0\end{array}\right)$. The diagonal matrices form a Cartan subalgebra $\mf h\sbs\mf{sl}_{2n}(\co)$, and the restriction $\mf h_1:=\mf h\cap \mf{sp}_{2n}(\co)$ is a Cartan subalgebra of $\mf{sp}_{2n}(\co)$. Here $\mf h_1$ consists of diagonal matrices of the form $e_i-e_{i+k}$ for $i=1,\ld,k$. For $i=1,\ld,2n$, define $L_i:\mf h\ra\co$ by 
\[\left(\begin{array}{ccc}x_1&&\\&\ddots&\\&&x_{2n}\end{array}\right)\mapsto x_i.\]
The weight space $\mf h^*$ is the quotient of the free $\co$-vector space $\lan L_1,\ld,L_{2n}\ran$ by the relation $L_1+\cd+L_{2n}=0$. Upon restriction $r:\mf h^*\ra\mf h_1^*$, there are further relations $r(L_i+L_{i+n})=0$ for $i=1,\ld,n$. Then $\mf h_1$ has a basis $L_1',\ld,L_n'$, where $L_i'=r(L_i)$. With this notation, it is elementary to show that the weights of the isotropy representation $\mf{sp}_{2n}(\co)\ra\End(\mf p)$ are 
\[\pm(L_i'-L_j'), \>\>\>\pm (L_i'+L_j'),\>\>\> 0\]
for $i< j$. The multiplicity of 0 is $n-1$, and all other weights have multiplicity 1. Without loss of generality, the highest weight is $L_1'+L_2'$ (this follows the convention in \cite{fh}). If this representation is the restriction of a (necessarily irreducible) representation $\mf{sl}_{2n}(\co)\ra\End(V)$ with highest weight $\la$, then $r(\la)=L_1'+L_2'$. Then $\la$ has the form $\la=2L_1+u$ where $u$ is an integral element of $\ker r$. To be precise, 
\[\begin{array}{lll}
\la&=(L_1+L_2)+a_1(L_1+L_{n+1})+\cd+a_n(L_n+L_{2n})\\[2mm]
&=L_1+L_2+\sum_{i=1}^na_i(L_i+L_{n+i})+a_nL_n-\sum_{i=1}^{2n-1}a_n\>L_i\\[2mm]
&=(1+a_1-a_n)L_1+(1+a_2-a_n)L_2+\sum_{i=3}^n(a_i-a_n)L_i+\sum_{i=1}^{n-1}(a_i-a_n)L_{n+i}
\end{array}\]
In the second line above we have used the relation $\sum_{i=1}^{2n} L_i=0$. Since $\la$ is a nonnegative sum of fundamental weights, the coefficient on $L_i$ is at least the coefficient on $L_{i+1}$. Then 
\[1+a_1-a_n\ge1+a_2-a_n\ge a_3-a_n\ge\cd\ge a_{n-1}-a_n\ge0\ge a_1-a_n\ge\cd\ge a_{n-1}-a_n.\]
It follows that $a_i-a_n=0$ for $i=1,\ld,n-1$. Then in fact
\[\la=L_1+L_2.\]
However, the representation of $\mf{sl}_{2n}(\co)$ with highest weight $L_1+L_2$ has dimension $2n(2n-1)$, which is equal to $\dim\mf p=(n-1)(2n+1)$ for no values of $n$. This shows the extension does not exist.
\end{proof}

\subsection{Extending the isotropy representation for $G=\SO_{n,1}$ and $G=E_{6(-26)}$} 

Let $G=\SO_{n,1}$. Upon complexification we are led to study the isotropy representation of $\mf k=\mf{so}_n(\co)\sbs\mf{so}_{n+1}(\co)=\mf g$. The isotropy representation has dimension $\dim\mf g-\dim\mf k=n$. On the other hand, the smallest nontrivial representation of $\mf g$ has dimension $n+1$. Then the isotropy representation of $\mf k$ does not extend to $\mf g$. 

Let $G=E_{6(-26)}$. Upon complexification, we are led to study the isotropy representation of $\mf k= \mf f_4(\co)\sbs\mf e_6(\co)=\mf g$. The isotropy representation has dimension $\dim \mf g-\dim\mf f_4=78-52=26$. On the other hand, the smallest dimension of a nontrivial representation of $\mf g$ is 27. Hence the isotropy representation of $\mf k$ does not extend to $\mf g$.

\subsection{The extension problem when $G$ is complex}\label{sec:extendcplx} Let $G$ be a complex simple Lie group with maximal compact subgroup $K$. In this case we argue as follows. Suppose for a contradiction that the isotropy representation $\iota:\mf k\ra\End(\mf p)$ extends to a representation $\rho:\mf g\ra\End(\mf p)$. After complexifying we have a representation $\rho_\co:\mf g\op\mf g\ra\End(\mf p_\co)$ that has the following key properties:
\begin{enumerate}
\item $\rho_\co$ is irreducible.
\item The restriction of $\rho_\co$ to the diagonal $\mf g\sbs\mf g\op\mf g$ is the adjoint representation $\ad:\mf g\ra\End(\mf g)$. 
\item The restriction of $\rho_\co$ to the real form $\mf g_\re\sbs\mf g\op\mf g$ has a real structure (see below).
\end{enumerate}
Since $\rho_\co$ is irreducible, $\mf p_\co$ is isomorphic (as a representation of $\mf g\op\mf g$) to a tensor product $V_1\ot V_2$, where $V_1$ and $V_2$ are irreducible representations of $\mf g$. Since $\rho_\co$ extends the adjoint representation of $\mf g$, we must have $\dim(V_1)\cdot\dim(V_2)=\dim\mf g$. Finally, since $\rho_\co\rest{}{\mf g_\re}$ has a real structure, neither $V_1$ nor $V_2$ is the trivial representation. By considering the possible representations $V_1,V_2$ we conclude that $\dim V_1\cdot\dim V_2>\dim\mf g$. This is a contradiction, so the extension $\rho:\mf g\ra\End(\mf p)$ does not exist. 

To elaborate on this argument we need the following terminology. For a more detailed treatment see \cite{reallie}. 

\vspace{.1in}

\noindent{\bf Real forms and real structures on a representation.} 
Let $V$ be a complex vector space. A \emph{real structure} on $V$ is an anti-linear involution $S:V\ra V$. For any real structure the fixed vectors
\[V^S=\{v\in V: S(v)=v\}\] form a real vector space. For example, complex conjugation is a real structure on $V=\co^n$. 

Let $\mf g$ be a complex Lie algebra with underlying real Lie algebra $\mf g_\re$. A representation $\rho:\mf g_\re\ra\End(V)$ induces an obvious representation $\ov\rho:\mf g_\re\ra\End(\ov V)$ on the conjugate vector space $\ov V$. If $\rho$ and $\ov\rho$ are isomorphic, then $\rho$ is called \emph{self-conjugate}. 

A \emph{compatible real structure} for $\rho:\mf g_\re\ra\End(V)$ is a real structure $S:V\ra V$ so that $S\circ\rho(x)=\rho(x)\circ S$ for every $x\in\mf g_\re$. A compatible real structure induces a real representation $\mf g_\re\ra\End(V^S)$. In addition, a compatible real structure defines an isomorphism between $\rho$ and $\ov\rho$. In other words, if the representation $\rho$ has a compatible real structure, then $\rho$ is self-conjugate.  

A real structure $S:\mf g\ra\mf g$ on a complex Lie algebra defines a real Lie algebra $\mf g_0:=\mf g^S$, which is called the \emph{real form} corresponding to $S$. 

We proceed to the statement of Proposition \ref{prop:selfconj}. For a complex Lie group $\mf g$, the underlying real Lie group $\mf g_\re$ is a real form of $\mf g\op\mf g$. Denote by $\aut(\Pi)$ the group of permutations of the \emph{fundamental weights} of $\mf g\op\mf g$ that preserve the \emph{Cartan matrix} of $\mf g$, and let $\nu\in\aut(\Pi)$ be the \emph{Weyl involution} associated to the real form $\mf g_\re\sbs\mf g\op\mf g$ (see \cite{reallie}). The following lemma gives a criterion to determine if a complex representation of $\mf g_\re$ has a real structure. This is a special case of Theorem 3 in Ch.\ 8 of \cite{reallie}.

\begin{prop}\label{prop:selfconj}
Let $\rho_1:\mf g\ra\emph{\aut}(V_1)$ and $\rho_2:\mf g\ra\emph{\aut}(V_2)$ be irreducible (complex) representations with highest weights $\Lambda_1,\Lambda_2$, respectively. Let $V=V_1\ot V_2$ and $\rho_1\ot\rho_2:\mf g\op\mf g\ra\emph{\aut}(V)$ be the induced representation. Then the restriction $\rho_1\ot\rho_2\rest{}{\mf g_\re}$ has a real structure if and only if $\nu(\Lambda_1)=\Lambda_2$.
\end{prop}

\noindent We are finally ready to address the extension problem.
\begin{prop}
Let $G$ be a simple, complex Lie group. 
Let $K\sbs G$ be a maximal compact subgroup. Then the isotropy representation $\iota: \mf k\ra\emph{\End}(\mf p)$ does not extend to a representation $\mf g_\re\ra\emph{\End}(\mf p)$. 
\end{prop}

\begin{proof}
Suppose there exists a representation $\rho:\mf g_\re\ra\End(\mf p)$ so that $\rho\rest{}{\mf k}=\iota$. We complexify this situation. Observe that $\mf k_\co\simeq\mf g$ and the complexification of the isotropy representation $\iota_\co:\mf k_\co\ra\End(\mf p_\co)$ is isomorphic to the adjoint representation $\ad:\mf g\ra\End(\mf g)$. Furthermore, $(\mf g_\re)_\co\simeq\mf g\op\mf g$, and the complexification of the inclusion $\mf k\hra\mf g_\re$ is the diagonal map $\De:\mf g\ra\mf g\op\mf g$. Then if the extension $\rho$ exists, there exists a extension $\rho_\co$ making the following diagram commute.
\begin{equation}\label{diag:ext1}
\begin{gathered}
\begin{xy}
(0,0)*+{\mf g}="A";
(0,-15)*+{\mf g\op\mf g}="B";
(30,0)*+{\End(\mf p_\co)}="C";
{\ar"A";"B"}?*!/^3mm/{\De};
{\ar@{-->} "B";"C"}?*!/^3mm/{\rho_\co};
{\ar "A";"C"}?*!/_3mm/{\ad};
\end{xy}
\end{gathered}
\end{equation}
(It is worth noting that there \emph{is} an obvious representation $\rho_\co$ that makes Diagram \ref{diag:ext1} commute. Let $p_1:\mf g\op\mf g\ra\mf g$ be projection to the first factor, and let $\rho_\co=\ad\circ p_1$. 
However, the representation $\ad\circ p_1$ cannot be the complexification of $\rho$ because it does not restrict in the correct way to $\mf g_\re\sbs\mf g\op\mf g$.)

Note that $\rho_\co$ must be irreducible because  $\ad$ is irreducible. Then $\mf p_\co$ is isomorphic (as a representation of $\mf g\op\mf g$) to $V_1\ot V_2$ for two irreducible representations $V_1,V_2$ of $\mf g$. Furthermore,
\begin{equation*}\label{eqn:dim}\dim(V_1)\cdot\dim(V_2)=\dim(V_1\ot V_2)=\dim\mf p_\co=\dim\mf g.\end{equation*}

Let $\Lambda_i\in\mf h^*$ be the highest weight of $V_i$. Then $(\Lambda_1,\Lambda_2)\in\mf h^*\op\mf h^*$ is the highest weight of $V_1\ot V_2$. The fact that $\rho_\co$ is the complexification of a solution $\rho$ to the real version of the extension problem implies that the restriction of $\rho_\co$ to $\mf g_\re\sbs\mf g\op\mf g$ has a compatible real structure. By Proposition \ref{prop:selfconj}, this implies that $\Lambda_2=\nu(\Lambda_1)$. Then either $V_1$ and $V_2$ are both trivial or both nontrivial. Clearly they cannot both be trivial, so they are nontrivial. Let $d$ denote the smallest dimension of a nontrivial representation of $\mf g$. Then 
\[\dim(V_1)\cdot\dim(V_2)\ge d^2.\] On the other hand, one checks that $\dim \mf g<d^2$ in each case (see table below). This contradiction implies that the extension $\rho:\mf g\ra\End(\mf p)$ cannot exist. 
\end{proof}

\begin{center} 
\begin{tabular}{c|c|ccc}
$\mf g$&$\dim\mf g$&$d$\\[1mm]
\hline &&\\ [-1.5ex]
$\mf{sl}_n(\co)$&$n^2-1$&$n$\\[2mm]
$\mf{sp}_{2n}(\co)$&$n(2n+1)$&$2n$\\[2mm]
$\mf{so}_n(\co)$&$n(n-1)/2$&$n$\\[2mm]
$\mf g_2(\co)$&14&7\\[2mm]
$\mf f_4(\co)$&52&26\\[2mm]
$\mf e_6(\co)$&78&27\\[2mm]
$\mf e_7(\co)$&133&56\\[2mm]
$\mf e_8(\co)$&248&248
\end{tabular}
\end{center}




\section{Point-pushing and the Zimmer program}\label{sec:zimmer}

Let $G$ be a semisimple real Lie group without compact factors and $\re$-rank $\ge2$, and let $\Ga\sbs G$ be a lattice. The linear actions of $\Ga$ are essentially classified by the Margulis Superrigidity Theorem \ref{thm:margulis}. It follows from the classification that there is a smallest dimension of a nontrivial linear representation of $\Ga$. 
The Zimmer program is concerned with classifying smooth actions of $\Ga$ on manifolds. Zimmer's conjecture states that there is a smallest dimension of a manifold $M$ on which $\Ga$ acts, and this dimension is computed explicitly from $G$ (see \cite{fisher}). 

In Theorem \ref{thm:ppzimmer} below we illustrate the Zimmer conjecture in a special case. This example was shown to the author by S.\ Weinberger. Take $\Ga\sbs G$ as above, and let $d$ be the smallest dimension of a nontrivial linear representation of $G$ (if $G$ is simple, this is the dimension of the standard representation). Let $M$ be any 4-manifold with $\pi_1(M)=\Ga$; this can be done because $\Ga$ is finitely presented (see \cite{witte-morris}). 

\begin{thm}\label{thm:ppzimmer}
Let $G$, $\Ga$, $M$, and $d$ be as above. If $d\ge 5$ then \emph{Push} is not realized by diffeomorphisms. 
\end{thm}

Theorem \ref{thm:ppzimmer} is an easy consequence of the Margulis Superrigidity Theorem \ref{thm:margulis} and the Thurston Stability Theorem \ref{thm:thurstonstability}.
\begin{thm}[Thurston stability \cite{thurston_stability}]\label{thm:thurstonstability}
Let $\Lambda$ be a finitely generated group with $H^1(\Lambda;\re)=0$. Let $M$ be a connected manifold. Assume that $\Lambda$ acts on $M$ by $C^1$ diffeomorphisms. If there exists a global fixed point $*\in M$ and each $\la\in \Lambda$ acts trivially on $T_*M$, then $\Lambda$ acts trivially on $M$. 
\end{thm}

Note that if $\Ga$ is a lattice in a semisimple real Lie group, then $\Ga$ is finitely generated (see \cite{witte-morris}). In addition, if $G$ has rank at least $2$, then $\Ga$ has no continuous map to $\re$ (the image of any homomorphism $\Ga\ra\gl_n(\re)$ has semisimple Zariski closure). Then $\Ga$ satisfies the hypotheses of Theorem \ref{thm:thurstonstability}. 

\begin{proof}[Proof of Theorem \ref{thm:ppzimmer}]
Suppose, for a contradiction, that Push is realized $\vp:\Ga\ra\diff(M,*)$. Then $\Ga$ acts on the tangent space at the fixed point
\[\al:\Ga\ra\aut(T_*M)\simeq \gl_4\re.\]
As discussed in Section \ref{sec:res}, the Zariski closure of $\al(\Ga)$ is either finite or is a product of factors in $\prod H^\si$. Note that the smallest nontrivial representation of $H^\si$ has dimension $d\ge 5$, so $H^\si$ cannot be contained in the Zariski closure of $\al(\Ga)$. Then $\al(\Ga)$ must be finite, and $\Ga'=\ker(\al)$ is finite index in $\Ga$. By definition $\Ga'$ acts trivially on $T_*(M)$, so $\Ga'$ acts trivially on $M$ by Theorem \ref{thm:thurstonstability}. Since $\vp$ is a lift of $\text{Push}$, this implies that $\Ga'$ is in the kernel of $\text{Push}:\Ga\ra\pi_0(\diff(M,*))$. On the other hand the kernel of Push is contained in the center of $\Ga$, which is finite. This is a contradiction, so $\text{Push}$ is not realized by diffeomorphisms. 
\end{proof}

\bibliographystyle{plain}
\bibliography{tpbib}

\begin{thebibliography}{10}

\bibitem{bgs}
Werner Ballmann, Mikhael Gromov, and Viktor Schroeder.
\newblock {\em Manifolds of nonpositive curvature}, volume~61 of {\em Progress
  in Mathematics}.
\newblock Birkh\"auser Boston, Inc., Boston, MA, 1985.

\bibitem{bcs}
Mladen Bestvina, Thomas Church, and Juan Souto.
\newblock Some groups of mapping classes not realized by diffeomorphisms.
\newblock {\em Comment. Math. Helv.}, 88(1):205--220, 2013.

\bibitem{bw}
Jonathan Block and Shmuel Weinberger.
\newblock On the generalized {N}ielsen realization problem.
\newblock {\em Comment. Math. Helv.}, 83(1):21--33, 2008.

\bibitem{bh}
A.~Borel and F.~Hirzebruch.
\newblock Characteristic classes and homogeneous spaces. {I}.
\newblock {\em Amer. J. Math.}, 80:458--538, 1958.

\bibitem{borel:toplie}
Armand Borel.
\newblock Topology of {L}ie groups and characteristic classes.
\newblock {\em Bull. Amer. Math. Soc.}, 61:397--432, 1955.

\bibitem{bg}
Michelle Bucher and Tsachik Gelander.
\newblock The generalized {C}hern conjecture for manifolds that are locally a
  product of surfaces.
\newblock {\em Adv. Math.}, 228(3):1503--1542, 2011.

\bibitem{ee}
Clifford~J. Earle and James Eells.
\newblock A fibre bundle description of {T}eichm\"uller theory.
\newblock {\em J. Differential Geometry}, 3:19--43, 1969.

\bibitem{fm}
Benson Farb and Dan Margalit.
\newblock {\em A primer on mapping class groups}, volume~49 of {\em Princeton
  Mathematical Series}.
\newblock Princeton University Press, Princeton, NJ, 2012.

\bibitem{fisher}
D.~Fisher.
\newblock {Groups acting on manifolds: around the Zimmer program}.
\newblock http://arxiv.org/pdf/0809.4849v2.pdf, December 2008.

\bibitem{frankshandel}
John Franks and Michael Handel.
\newblock Global fixed points for centralizers and {M}orita's theorem.
\newblock {\em Geom. Topol.}, 13(1):87--98, 2009.

\bibitem{fh}
William Fulton and Joe Harris.
\newblock {\em Representation theory}, volume 129 of {\em Graduate Texts in
  Mathematics}.
\newblock Springer-Verlag, New York, 1991.
\newblock A first course, Readings in Mathematics.

\bibitem{giansiracusa}
Jeffrey Giansiracusa.
\newblock The diffeomorphism group of a {$K3$} surface and {N}ielsen
  realization.
\newblock {\em J. Lond. Math. Soc. (2)}, 79(3):701--718, 2009.

\bibitem{hatcher}
Allen Hatcher.
\newblock {\em Algebraic topology}.
\newblock Cambridge University Press, Cambridge, 2002.

\bibitem{helgason2}
Sigurdur Helgason.
\newblock {\em Differential geometry and symmetric spaces}.
\newblock Pure and Applied Mathematics, Vol. XII. Academic Press, New
  York-London, 1962.

\bibitem{kt_flat}
Franz Kamber and Philippe Tondeur.
\newblock {\em Flat manifolds}.
\newblock Lecture Notes in Mathematics, No. 67. Springer-Verlag, Berlin-New
  York, 1968, 1968.

\bibitem{Kerckhoff}
Steven~P. Kerckhoff.
\newblock The {N}ielsen realization problem.
\newblock {\em Ann. of Math. (2)}, 117(2):235--265, 1983.

\bibitem{ks}
Robion~C. Kirby and Laurence~C. Siebenmann.
\newblock {\em Foundational essays on topological manifolds, smoothings, and
  triangulations}.
\newblock Princeton University Press, Princeton, N.J.; University of Tokyo
  Press, Tokyo, 1977.
\newblock With notes by John Milnor and Michael Atiyah, Annals of Mathematics
  Studies, No. 88.

\bibitem{Margulis}
G.~A. Margulis.
\newblock {\em Discrete subgroups of semisimple {L}ie groups}, volume~17 of
  {\em Ergebnisse der Mathematik und ihrer Grenzgebiete (3) [Results in
  Mathematics and Related Areas (3)]}.
\newblock Springer-Verlag, Berlin, 1991.

\bibitem{markovic}
Vladimir Markovic.
\newblock Realization of the mapping class group by homeomorphisms.
\newblock {\em Invent. Math.}, 168(3):523--566, 2007.

\bibitem{milnor_liediscrete}
J.~Milnor.
\newblock On the homology of {L}ie groups made discrete.
\newblock {\em Comment. Math. Helv.}, 58(1):72--85, 1983.

\bibitem{ms}
John~W. Milnor and James~D. Stasheff.
\newblock {\em Characteristic classes}.
\newblock Princeton University Press, Princeton, N. J.; University of Tokyo
  Press, Tokyo, 1974.
\newblock Annals of Mathematics Studies, No. 76.

\bibitem{morita_nonlifting}
Shigeyuki Morita.
\newblock Characteristic classes of surface bundles.
\newblock {\em Invent. Math.}, 90(3):551--577, 1987.

\bibitem{witte-morris}
D.~Witte Morris.
\newblock {Introduction to arithmetic groups}.
\newblock In preparation, 2012.

\bibitem{reallie}
Arkady~L. Onishchik.
\newblock {\em Lectures on real semisimple {L}ie algebras and their
  representations}.
\newblock ESI Lectures in Mathematics and Physics. European Mathematical
  Society (EMS), Z\"urich, 2004.

\bibitem{on}
P.~Ontaneda.
\newblock {Pinched smooth hyperbolization}.
\newblock http://arxiv.org/pdf/1110.6374v1.pdf, October 2011.

\bibitem{Steenrod}
Norman Steenrod.
\newblock {\em The {T}opology of {F}ibre {B}undles}.
\newblock Princeton Mathematical Series, vol. 14. Princeton University Press,
  Princeton, N. J., 1951.

\bibitem{thurston_stability}
William~P. Thurston.
\newblock A generalization of the {R}eeb stability theorem.
\newblock {\em Topology}, 13:347--352, 1974.

\bibitem{tshishiku2}
Bena Tshishiku.
\newblock {Pontryagin classes of locally symmetric manifolds}.
\newblock arxiv:1404.1115. Submitted, Mar. 2014.

\end{thebibliography}

\end{document}